\documentclass[11pt]{article}

\usepackage{latexsym,amsmath,amssymb,mathrsfs,graphicx,hyperref,mathpazo,soul,color,amsthm,marginnote}
\usepackage{float}
\usepackage{mathtools}
\usepackage{rotating}

\usepackage{url}
\def\tr{{\raise0pt\hbox{$\scriptscriptstyle\top$}}}

\newtheorem{theorem}{Theorem}[section]

\newtheorem{proposition}[theorem]{Proposition}
\newtheorem{lemma}[theorem]{Lemma}

\title{\vspace{-1.5 cm}{\bf Diophantine equations with three monomials}}

\author{Bogdan Grechuk\footnote{School of Computing and Mathematical Sciences, University of Leicester, LE1 7RH, UK; bg83@leicester.ac.uk}
\and
Tetiana Grechuk
\and
Ashleigh Wilcox
}

\begin{document}
	
\maketitle

\begin{abstract}
%% Text of abstract
We present a general algorithm for solving all two-variable polynomial Diophantine equations consisting of three monomials. 
Before this work, even the existence of an algorithm for solving the one-parameter family of equations $x^4+axy+y^3=0$ has been an open question. 
We also present an elementary method that reduces the task of finding all integer solutions to a general three-monomial equation to the task of finding primitive solutions to equations with three monomials in disjoint variables. We identify a large class of three-monomial equations for which this method leads to a complete solution. Empirical data suggests that this class contains $100\%$ of three-monomial equations as the number of variables goes to infinity.   
\end{abstract}

%%Graphical abstract
%\begin{graphicalabstract}
%%\includegraphics{grabs}
%\end{graphicalabstract}

%%Research highlights
%\begin{highlights}
%\item Research highlight 1
%\item Research highlight 2
%\end{highlights}



\section{Introduction}
Diophantine equations are polynomial equations with integer coefficients, usually in two or more variables, for which we are interested in finding integer solutions. This topic is almost as old as  mathematics. The name ``Diophantine'' refers to Diophantus of Alexandria, who studied such equations in the 3rd century. In 1920, Dickson \cite{dickson1920history} published an overview of works in this area up to that date, with essentially no proofs, just an annotated bibliography of results, and it still took the whole volume with 800 pages. Any similar overview today would require dozens of volumes, because Diophantine equations attract mathematicians of all levels, with hundreds of papers and many books published every year. For an introduction to the subject, we recommend the 1969 book of Mordell \cite{mordell1969diophantine}, and more recent excellent books of Cohen \cite{cohen2008number} and Andreescu \cite{andreescu2010introduction}.

In 1900, Hilbert \cite{hilbert1902mathematical} presented a list of $23$ mathematical problems for the next centuries. Hilbert's 10th problem asks for a general method for determining whether a given Diophantine equation has any integer solution. This task looks much easier than describing all integer solutions. Nevertheless, Matiyasevich \cite{matijasevic1970enumerable}, building on an earlier work of Davis, Putnam and Robinson \cite{davis1961decision}, proved in 1970 that the answer to Hilbert’s 10th problem is negative, and no such general method exists.

While solving all Diophantine equations is impossible, researchers focus on developing methods for solving restricted families of equations. For example, one may restrict the degree of the equation or the number of variables. A notable result in this direction is a deep theorem of Grunewald and Segal \cite{grunewald1981solve, grunewald2004integer}, that proves the existence of a finite algorithm which, given a quadratic equation in any number of variables, determines whether its integer solution set is finite or infinite, and, in the former case, lists all the solutions. Also, a combination of the main theorems in \cite{baker1970integer}, \cite{poulakis1993points}, and \cite{poulakis2002solving} imply the existence of an algorithm for solving all cubic equations in $2$ variables. However, there are no known methods for solving cubic equations in $3$ variables, or quartic equations in $2$ variables, see \cite{gasarch2021hilbertb}.

Another important parameter of a Diophantine equation which we may restrict is the number of its monomials.  Equations having at most two monomials are easy to solve \cite{bell1933reciprocal,ward1933type}, while equations with three monomials may be very difficult in general. A canonical example of a difficult three-monomial equation is of course Fermat's Last Theorem (FLT), resolving equations of the form $x^n+y^n=z^n$ for $n\geq 3$. It can be viewed as either an exponential Diophantine equation, or an infinite family of polynomial equations. While FLT is now proved by Wiles \cite{wiles1995modular}, its natural generalization  
\begin{equation}\label{eq:genFermat}
ax^n + by^m = cz^k, 
\end{equation}
which is called the generalized Fermat equation, remains the subject of current deep research.
% \cite{bennett2016generalized}. 
In most works, the authors are interested to find the primitive solutions to \eqref{eq:genFermat}, that is, ones satisfying $\text{gcd}(x,y,z)=1$. In 1998, Beukers \cite{beukers1998diophantine} proved that if $\frac{1}{n}+\frac{1}{m}+\frac{1}{k}>1$, then the set of all primitive solutions to \eqref{eq:genFermat} can be covered by a finite number of polynomial parametrizations, and there is an algorithm that, given $a,b,c,n,m,k$ with $\frac{1}{n}+\frac{1}{m}+\frac{1}{k}>1$, produces such parametrizations. In the case $\frac{1}{n}+\frac{1}{m}+\frac{1}{k}=1$, the problem reduces to finding rational points on certain elliptic curves. There is no known algorithm for solving this problem in general, but there are methods that can find rational points on many specific curves, so individual equations \eqref{eq:genFermat} with $\frac{1}{n}+\frac{1}{m}+\frac{1}{k}=1$ and not-too-large $a,b,c$ can usually be solved by standard methods. If $\frac{1}{n}+\frac{1}{m}+\frac{1}{k}<1$, then Darmon and Granville \cite{darmon1995equations} proved that, for any specific $a,b,c,n,m,k$, equation \eqref{eq:genFermat} has only a finite number of primitive solutions, which raises the problem of listing these solutions explicitly. The most studied special case is $a=b=c=1$, so that equation \eqref{eq:genFermat} reduces to $x^n+y^m=z^k$. Many recent papers solve this equation for various triples $(n,m,k)$, see e.g. \cite{bennett2016generalized} for a survey. A recent example is the resolution of case $(2,3,11)$ in \cite{freitas2020generalized}, conditional on the generalized Riemann hypothesis. 

In this work, we are interested in describing \emph{all} integer solutions to a given three-monomial equation, not only the primitive ones. Our first main result is the general algorithm for solving all three-monomial equations in two variables. Our proof uses deep classical results in the theory of two-variable Diophantine equations, such as Baker's resolution of superelliptic equations \cite{baker1969bounds} and Walsh's effective version of Runge's theorem \cite{walsh1992quantitative}, but is otherwise elementary. As a very special case, our algorithm resolves a family of equations
$$
x^4+axy+y^3=0,
$$
which Masser \cite{MR4396165} used as an example of a simple-looking family that ``does not seem to be effectively solvable''. 

We also present a method that reduces an arbitrary three-monomial equation, in any number of variables, to a finite number of equations that have three monomials in disjoint variables.
%given an arbitrary three-monomial equation in any number of variables, reduces it to a finite number of equations that have three monomials in disjoint variables.
In many examples, the resulting equations are either easier than the original one, or are already solved in the literature. As an example, we present a family of equations that our method reduces to FLT. As another example, we give a description of \emph{all} integer solutions to the equation \eqref{eq:genFermat}, provided that at least one of the exponents $n,m,k$ is coprime with the other two. In addition, we identify a simple and easy-to-check sufficient condition that guarantees that a given three-monomial equation is solvable by our method, and empirically check that the percentage of three-monomial equations satisfying this condition approaches $100\%$ when the number of variables goes to infinity.
%This turns out to be much easier (and the answer is much shorter) than the description of \emph{primitive} solutions to this equation in  \cite{edwards2005platonic}. 

This work is organised into four sections. Section \ref{sec:2var3mon} overviews some deep results in the theory of two-variable Diophantine equations, and uses  them to prove the existence of a general algorithm for solving all two-variable equations consisting of three monomials. Section \ref{sec:gen3mon} studies three-monomial equations in any number of variables. Section \ref{sec:concl} concludes the work. 
%Appendix \ref{app:2mon} presents an algorithm for solving all two-monomial equations in any number of variables.

\section{Two-variable equations with three monomials}\label{sec:2var3mon} 

\subsection{Some solvable families of two variable equations}\label{sec:2var}

A monomial in variables $x_1,\dots,x_n$ with integer coefficient is any expression of the form $ax_1^{k_1}\dots x_n^{k_n}$, where $a$ is a non-zero integer and $k_1, \dots, k_n$ are non-negative integers. A polynomial $P(x_1,\dots,x_n)$ is the sum of any (finite) number of such monomials. A polynomial Diophantine equation is an equation of the form 
\begin{equation}\label{eq:diofgen}
P(x_1, \dots, x_n) = 0.
\end{equation} 
A natural way to classify Diophantine equations is by the number of variables. Equations in one variable 
\begin{equation}\label{eq:onevar}
\sum_{i=m}^d a_i x^i = 0, \quad \text{where} \,\, a_m, a_{m+1}, \dots, a_d \,\, \text{are integers such that} \,\, a_m\neq 0 \,\, \text{and} \,\, a_d \neq 0,
\end{equation} 
are easy to solve in integers, because every non-zero solution $x$ must be a divisor of $a_m$. In fact, \eqref{eq:onevar} is also easy to solve in rationals, because, if $x=\frac{p}{q}$ is a non-zero irreducible fraction, then, by multiplying  \eqref{eq:onevar} by $q^d$ and dividing by $p^m$, we obtain
$$
	\sum_{i=m}^d a_i p^{i-m} q^{d-i} = 0.
$$
The first term of this sum is $a_m q^{d-m}$. Because all other terms are divisible by $p$, and $p,q$ are coprime, this implies that $p$ must be a divisor of $a_m$. By similar argument, the last term $a_d p^{d-m}$ is divisible by $q$, hence $q$ is a divisor of $a_d$. Because $a_m$ and $a_d$ have a finite number of divisors, there is only a finite number of candidate fractions $\frac{p}{q}$ to check. In fact, there is a much faster algorithm \cite{lenstra2012finding} for listing all rational solutions to \eqref{eq:onevar} that works in time polynomial in the length of the input.

A general algorithm for finding integer or rational solutions for all equations in two variables is currently unknown, and this is a major open problem.

We may also classify equations by the number of monomials. There are algorithms for describing all integer solutions to an arbitrary two-monomial equation in any number of variables, see \cite{bell1933reciprocal,ward1933type}, hence the first open case are three-monomial equations.
%Because we were not able to find a reference, we present an algorithm in Appendix \ref{app:2mon}. 
Such equations can be very difficult in general, as witnessed by FLT and its generalizations. 

In this section, we restrict the number of variables and the number of monomials \emph{simultaneously}, and present a general algorithm for solving all two-variable three-monomial Diophantine equations. Before proving the main result, we first overview some well-known solvable families of two-variable Diophantine equations. 
%We start with a deep theorem of Baker that imply the existence of an algorithm for solving all equations of the form
%\begin{equation}\label{eq:aym}
%a y^m = P(x),
%\end{equation}
%where $m\geq 0$ and $a$ are integers, and $P(x)$ is a polynomial with integer coefficients. Equations in the form \eqref{eq:aym} are known as superelliptic equations. 
%They are trivial for $m=0$ and $m=1$, so we start the discussion with the first non-trivial case $m=2$, in which case \eqref{eq:aym} are known as hyperelliptic equations.
 
In 1969, Baker \cite{baker1969bounds} proved the following theorem.
\begin{theorem}\label{th:ellBaker1}
There is an algorithm for listing all integer solutions (there are finitely many of them) of the equation 
\begin{equation}\label{eq:hyperell}
y^m = P(x) = a_n x^n + \dots + a_1 x + a_0,
\end{equation}
provided that all $a_i$ are integers, $a_n\neq 0$, and either (i) $m=2$ and $P(x)$ possesses at least three simple (possibly complex) zeros, or (ii) $m\geq 3$, $n\geq 3$, and $P(x)$ has at least two simple zeros.
\end{theorem} 

%As a special case, Theorem \ref{th:ellBaker1} implies the existence of an algorithm that, given integers $a,b,c$ and non-negative integers $n,m,k$, outputs a description of the solution set to the three-monomial equation
%\begin{equation}\label{eq:cympaxnpb}
%a y^m = b x^n + c x^k.
%\end{equation}
%Indeed, by symmetry, we may assume that $n\geq k$, hence \eqref{eq:cympaxnpb} can be rewritten as 
%$$
%ay^m = x^k(b x^{n-k}+c).
%$$

As a special case, Theorem \ref{th:ellBaker1} implies the existence of an algorithm that, given integers $a,b,c$ and non-negative integers $n,m$, outputs a description of the solution set to the three-monomial equation
\begin{equation}\label{eq:cympaxnpb}
a y^m = b x^n + c.
\end{equation}
Indeed, by swapping $x$ and $y$ if necessary, we may assume that $m\leq n$. If $ab=0$ or $m=0$, then \eqref{eq:cympaxnpb} is an equation in one variable. If $c=0$, then \eqref{eq:cympaxnpb} is a two-monomial equation and can be solved by the algorithms in  \cite{bell1933reciprocal,ward1933type}. Otherwise we may consider $|a|$ cases $x=|a|z, x=|a|z+1, \dots, x=|a|z+(|a|-1)$, where $z$ is a new integer variable, and in each case, equation \eqref{eq:cympaxnpb} is either not solvable because the right-hand side is not divisible by $a$, or, after cancelling $a$, reduces to an equation of the form
\begin{equation}\label{eq:cympaxnpbred}
y^m = P(z),
\end{equation}
where $P(z)=(b(|a|z+r)^n+c)/a$ for some $0\leq r < |a|$ such that $br^n+c$ is divisible by $a$.
%an equation in the same form \eqref{eq:cympaxnpb} but with $a=1$. 
Then if $m=1$ we can take $z=u$ an arbitrary integer, and express $y$ as $y=P(u)$. If $m=n=2$, then the solution methods for equation \eqref{eq:cympaxnpb} are well-known, see e.g. \cite[Theorem 4.5.1]{andreescu2015quadratic}. Finally, if $m\geq 2$, $n\geq 3$, and $c\neq 0$, then polynomial $P(z)$ in \eqref{eq:cympaxnpbred} has $n\geq 3$ simple zeros, and equation \eqref{eq:cympaxnpbred} is covered by Theorem \ref{th:ellBaker1}. 

%As it is common for Baker's method, 
The original proof of Theorem \ref{th:ellBaker1}  works by establishing a (very large) upper bound on the size of possible solutions, and the algorithm checks all pairs $(x,y)$ up to this bound and is therefore completely impractical. In 1998, Bilu and Hanrot \cite{MR1631771} developed a more involved but much faster version of this algorithm. So, given any equation of the form \eqref{eq:cympaxnpb}, one may follow the lines of \cite{MR1631771} and solve the equation. Of course, this may require a lot of time, effort, and knowledge of non-elementary mathematics, but at least it is possible in finite time. In the special case $m=n$, \eqref{eq:cympaxnpb} belongs to a family of Thue equations for which there are easier and faster algorithms, see e.g. \cite{MR987566}.  

%Another important example of a solvable family of two-variable equations are the ones satisfying Runge's condition. This family of equations was first studied by Runge \cite{runge1887ueber} in 1887. We say that the equation
%$$
%P(x,y) = \sum_{i=0}^m \sum_{j=0}^n a_{ij} x^i y^j = 0,
%$$
%with integer coefficients $a_{ij}$ of degree $m>0$ in $x$ and $n>0$ in $y$, satisfies the \textbf{Runge's condition} if $P(x,y)$ is irreducible over ${\mathbb Q}$ and either
%\begin{itemize}
%\item[(C1)] there exists a coefficient $a_{ij}\neq 0$ of $P$ such that $ni+mj>mn$,
%
%or 
%
%\item[(C2)] the sum of all monomials $a_{ij}x^iy^j$ of $P$ for which $ni + mj = nm$ can be decomposed into a product of two non-constant coprime polynomials.
%\end{itemize} 

Another important example of a solvable family of two-variable equations are the ones satisfying Runge's condition. This family of equations was first studied by Runge \cite{runge1887ueber} in 1887. We say that the equation
$$
P(x,y) = \sum_{i=0}^n \sum_{j=0}^m a_{ij} x^i y^j = 0,
$$
with integer coefficients $a_{ij}$ of degree $n>0$ in $x$ and $m>0$ in $y$, satisfies the \textbf{Runge's condition} if $P(x,y)$ is irreducible over ${\mathbb Q}$ and either
\begin{itemize}
	\item[(C1)] there exists a coefficient $a_{ij}\neq 0$ of $P$ such that $nj+mi>mn$,
	
	or 
	
	\item[(C2)] the sum of all monomials $a_{ij}x^iy^j$ of $P$ for which $nj + mi = nm$ can be decomposed into a product of two non-constant coprime polynomials.
\end{itemize}

In 1887, Runge \cite{runge1887ueber} proved that if a polynomial $P$ satisfies this condition, then equation $P=0$ has at most finitely many integer solutions and outlined (somewhat informally) a method for finding these solutions. In 1992, Walsh \cite{walsh1992quantitative} developed an effective upper bound for the size of possible solutions, which definitely implies the existence of a formal algorithm for listing all the solutions. See \cite{beukers2005implementation} for a practical version of such algorithm.

\begin{theorem}\label{th:Runge} \cite{runge1887ueber, walsh1992quantitative}
There is an algorithm that, given any polynomial $P$ satisfying the  Runge's condition, determines all integer solutions of equation $P=0$.
\end{theorem}

\subsection{An algorithm for solving two-variable three-monomial equations}\label{sec:3mon2var}

In this section we present a general method for solving all three-monomial equations in two variables. Before this work, it was a belief that some of the equations from this family are difficult. For example, Masser \cite{MR4396165} presented a one-parameter family 
\begin{equation}\label{eq:x4paxypy3}
x^4+axy+y^3=0
\end{equation}
as an example of a family that ``does not seem to be effectively solvable''. We start this section by presenting a solution to this particular family of equations.

\begin{proposition}\label{prop:x4paxypy3}
There exists an algorithm that, given any integer $a$, outputs all integer solutions to equation \eqref{eq:x4paxypy3}. 
\end{proposition}
\begin{proof}
If $a=0$, \eqref{eq:x4paxypy3} is an easy two-monomial equation, hence we can assume that $a \neq 0$. If $xy=0$, then \eqref{eq:x4paxypy3} implies that $x=y=0$. Let $(x,y)$ be any integer solution to \eqref{eq:x4paxypy3} with $xy\neq 0$, and let $d$ be the integer of the same sign as $y$ such that $|d|=\text{gcd}(x,y)$. Then $x=dx_1$ and $y=dy_1$ for coprime integers $x_1,y_1$ such that $x_1\neq 0$ and $y_1>0$. Substituting this into \eqref{eq:x4paxypy3} and cancelling $d^2$, we obtain the equation
$$
d^2 x_1^4 + ax_1y_1 + d y_1^3 = 0,
$$
from which it is clear that $dy_1^3$ is divisible by $x_1$. Because $\text{gcd}(x_1,y_1)=1$, this implies that $d=kx_1$ for some non-zero integer $k$. Then
\begin{equation}\label{eq:x4paxypy3aux}
0=(kx_1)^2 x_1^4 + a x_1 y_1 + (kx_1)y_1^3 = x_1(k^2x_1^5+ay_1+ky_1^3).
\end{equation}
The last equation implies that $k^2x_1^5$ is divisible by $y_1$. Because $x_1$ and $y_1$ are coprime, this implies that $k^2 = y_1 z$ for some integer $z>0$. Let $u=\text{gcd}(y_1,z)>0$. Then integers $y_1/u$ and $z/u$ are positive, coprime, and their product $(k/u)^2$ is a perfect square, hence $y_1/u=v^2$ and $z/u=w^2$ for some non-zero integers $v,w$, where we may assume that $w>0$ and $v$ has the same sign as $k$. Then $y_1=uv^2$, $k=uvw$, and \eqref{eq:x4paxypy3aux} implies that
$$ 
0 = (uvw)^2x_1^5+a(uv^2)+(uvw)(uv^2)^3 = uv^2(uw^2x_1^5 + a + wu^3v^5).
$$ 
Thus $uw^2x_1^5 + a + wu^3v^5=0$, which implies that $uw$ is a divisor of $a$. Because $a$ has a finite number of divisors, there are only finitely many pairs of positive integers $(u,w)$ satisfying this condition. 
%In fact, we can reduce the number of cases by assuming that $w>0$, because replacing $(u,v,w)$ with $(u,-v,-w)$ does not change $k$ and $y_1$. 
For each such pair $(u_i,w_i)$, we can find $x_1$ and $v$ from the equation
$$
u_i w_i ^2 x_1^5 + w_i u_i ^3 v^5= - a.
$$
These are the equations of the form \eqref{eq:cympaxnpb}, and they can be solved by Theorem \ref{th:ellBaker1}. With $x_1$ and $v$ at hand, we can compute $k=u_ivw_i$, $d=kx_1$, $y_1=u_i v^2$, $x=dx_1$ and $y=dy_1$.
\end{proof}

As an illustration, we have applied the described algorithm to solve equation \eqref{eq:x4paxypy3} for $1 \leq a \leq 100$. All solutions except the trivial $(x,y)=(0,0)$ are listed in Table \ref{tab:ntsx4paxypy3}.

\begin{table}
	\footnotesize
	\begin{center}
		\begin{tabular}{ | c | c |}
			\hline
			\textbf{$a$ } & \textbf{Non-trivial solutions to equation \eqref{eq:x4paxypy3}}\\\hline \hline
			$2$& $(-1,1),(2,-2)$ \\
			$6$& $(-6,-12),(-2,-4),(-2,2),(3,-3)$\\
			$8$&$(-4,-8)$\\
			$10$&  $(-2,4),(10,-20)$\\
			$12$&$(-3,3),(4,-4)$\\
			$20$&$(-4,4),(5,-5)$\\
			$24$&  $(-24,-72),(-4,8),(-3,-9),(12,-24)$\\
			$30$& $(-5,5),(-3,9),(6,-6),(30,-90)$\\
			$31$&  $(-62,-248),(-2,-8),(4,-2)$\\
			$33$&$(-4,2),(-2,8),(66,-264)$\\
			$35$&$(20,-50)$\\
			$42$&$(-21,-63),(-6,-18),(-6,6),(-6,12),(7,-7),(14,-28)$\\
			$54$&$(-18,-54),(-9,-27)$\\
			$56$&$(-7,7),(8,-8)$\\
			$60$& $(-60,-240),(-15,-45),(-12,-36),(-4,-16)$\\
			$64$&$(-8,16),(16,-32)$\\
			$66$&$(-6,18),(33,-99)$\\
			$68$& $(-4,16),(68,-272)$\\
			$69$&$(12,-18)$\\
			$72$&$(-8,8),(9,-9)$\\
			$87$&$(-58,-232),(-6,-24)$\\
			$88$&$(396,-2904)$\\
			$90$&$(-10,20),(-9,9),(10,-10),(18,-36)$\\\hline
		\end{tabular}
		\caption{Non-trivial solutions to $x^4+axy+y^3=0$ for $1 \leq a \leq 100$. \label{tab:ntsx4paxypy3}}
	\end{center} 
\end{table}

We next move to the analysis of the general equation in $2$ variables with $3$ monomials, that is, any equation of the form
\begin{equation}\label{eq:3mon2var}
a x^n y^q + b x^k y^l + c x^r y^m = 0,
\end{equation}
where $a,b,c$ are non-zero integers and $n,q,k,l,r,m$ are non-negative integers. The cases $x=0$ and $y=0$ can be checked separately, so we may focus on finding integer solutions such that $x\neq 0$ and $y\neq 0$. Then by cancelling the common term $x^{\min\{n,k,r\}}y^{\min\{q,l,m\}}$ if necessary, we may assume that $nkr = qlm = 0$. Without loss of generality, we may assume that $r=0$. If $ql \neq 0$, then $m=0$, and \eqref{eq:3mon2var} implies that $y$ is a divisor of $c$. By checking all possible divisors, we can solve \eqref{eq:3mon2var} easily. Hence, we may assume that $ql=0$, and, without loss of generality, $q=0$. Then equation \eqref{eq:3mon2var} reduces to
\begin{equation}\label{eq:axnpbxkylpcym}
a x^n + b x^k y^l + c y^m = 0
\end{equation}
%Now, if $k=l=0$, then equation \eqref{eq:axnpbxkylpcym} is of the form \eqref{eq:cympaxnpb}, hence we may assume that $k>0$ or $l>0$. 
%Next, we claim that we may assume that 
%\begin{equation}\label{eq:Rungexnpxkylpym}
%	nl+mk \leq mn.
%\end{equation}
%Indeed, we must have either $l=0$, or $k=0$, or $kl>0$. If $l=0$, then \eqref{eq:axnpbxkylpcym} reduces to $a x^n + b x^k + c y^m = 0$, hence \eqref{eq:Rungexnpxkylpym} can be ensured by swapping $(a,n)$ with $(b,k)$ if necessary. The case $k=0$ is similar. Finally, if $kl>0$ and $nl+mk>mn$ then \eqref{eq:axnpbxkylpcym} is trivial if $mn=0$, and otherwise \eqref{eq:axnpbxkylpcym} is irreducible and is covered by Theorem \ref{th:Runge}.

The following theorem is the first main result of this work.

\begin{theorem}\label{th:3mon2var} 
There is an algorithm that, given any three-monomial equation in $2$ variables, describes the set of all its integer solutions.
\end{theorem}

Before proving the theorem, we state and prove several easy lemmas. 

\begin{lemma}\label{le:xypzt}
	Every integer solution to equation
	\begin{equation}\label{eq:xypzt}
		xy-zt=0
	\end{equation} 
	is of the form
	\begin{equation}\label{eq:xypztsol}
		(x,y,z,t)=(uv,wr,uw,vr) \quad \text{for some integers} \,\, u,v,w,r. 
	\end{equation}
\end{lemma}
\begin{proof}
	For solutions with $x=0$ the statement is trivial. Let $(x,y,z,t)$ be any solution with $x\neq 0$, and let $u=\text{gcd}(x,z)>0$. Then $x=uv$ and $z=uw$ with $v,w$ coprime and $v\neq 0$. Then \eqref{eq:xypzt} implies that $uvy=uwt$, or $vy=wt$. Because $v$ and $w$ are coprime, $t$ must be divisible by $v$, so we can write $t=vr$ for some integer $r$. Then $vy=w(vr)$, hence $y=wr$, and \eqref{eq:xypztsol} follows. 
\end{proof}

\begin{lemma}\label{le:nxmyb}
	Let $n,m$ be non-negative integers that are not both zero, and $b$ be an arbitrary integer. If $b$ is not a multiple of $\text{gcd}(n,m)$, then equation
	\begin{equation}\label{eq:nxmyb}
		n x = m y + b, \quad\quad x,y \geq 0
	\end{equation}
    has no solutions in non-negative integers $x,y$. Otherwise the solutions to \eqref{eq:nxmyb} are given by
    \begin{equation}\label{eq:nxmybsol}
    	x = x_0 + \frac{m}{\text{gcd}(n,m)} u, \quad y = y_0 + \frac{n}{\text{gcd}(n,m)} u, \quad u \geq 0,
    \end{equation}
    where $(x_0,y_0)$ is the solution to \eqref{eq:nxmyb} with $x+y$ minimal. Note that $(x_0,y_0)$ depends only on $n,m,b$.
\end{lemma}
\begin{proof}
	If $b$ is not a multiple of $d=\text{gcd}(n,m)$, then $n x$ is divisible by $d$, while $m y + b$ is not, hence \eqref{eq:nxmyb} has no integer solutions. Now assume that $b$ is a multiple of $d$. The corresponding homogeneous equation
	$$
	n x = m y, \quad\quad x,y \geq 0
	$$
	has solutions
	\begin{equation}\label{eq:nxmysol}
		x = \frac{m}{\text{gcd}(n,m)} u, \quad y = \frac{n}{\text{gcd}(n,m)} u,
	\end{equation}
	where $u$ is an arbitrary non-negative integer. If $(x,y)$ is an arbitrary solution to \eqref{eq:nxmyb}, and $(x_0,y_0)$ is the solution to \eqref{eq:nxmyb} with $x+y$ minimal, then 
	$$
	n(x-x_0) = (my+b) - (m y_0+b) = m(y-y_0).
	$$
	Because $(x_0,y_0)$ is the solution with $x+y$ minimal, this implies that $x-x_0\geq 0$, $y-y_0\geq 0$, and, by \eqref{eq:nxmysol}, 
	$$
		x-x_0 = \frac{m}{\text{gcd}(n,m)} u, \quad y-y_0 = \frac{n}{\text{gcd}(n,m)} u.
	$$
	for some integer $u\geq 0$. 
\end{proof}

\begin{lemma}\label{le:main}
	Let $A,B,C$ be any integers such that $A+B+C=0$. Let $p$ be any prime, and let $A_p, B_p$ and $C_p$ be the exponents with which $p$ enters the prime factorizations of $A,B$ and $C$, respectively\footnote{In some works, the exponent with which $p$ enters the prime factorization of $x$ is denoted $v_p(x)$. In this work, we prefer a shorter notation.}. Then the smallest two of the integers $A_p, B_p, C_p$ must be equal.	
\end{lemma}
\begin{proof}
Let $M=\min\{A_p, B_p, C_p\}$, and let $M'$ be the second-smallest among these integers.
Then $A,B,C$ are all divisible by $p^M$, and
$$
\frac{A}{p^M} + \frac{B}{p^M} + \frac{C}{p^M} = 0.    
$$
However, if $M<M'$, then exactly two out of three integers $\frac{A}{p^M}, \frac{B}{p^M}, \frac{C}{p^M}$ are divisible by $p$, hence their sum is not divisible by $p$ and therefore cannot be $0$, a contradiction. Hence, $M=M'$.
\end{proof}

Now we are ready to prove Theorem \ref{th:3mon2var}.

\begin{proof} \emph{of Theorem \ref{th:3mon2var}}. 

We may assume that the equation is of the form \eqref{eq:axnpbxkylpcym} with non-zero integer coefficients $a,b,c$ and non-negative integer exponents $n,k,l,m$. Equation \eqref{eq:axnpbxkylpcym} is trivial if $m=n=0$, so we will assume that $m>0$ or $n>0$. Also, if $k=l=0$, then equation \eqref{eq:axnpbxkylpcym} is of the form \eqref{eq:cympaxnpb}, hence we may assume that $k>0$ or $l>0$. 
Next, we claim that we may assume that 
\begin{equation}\label{eq:Rungexnpxkylpym}
	nl+mk \leq mn.
\end{equation}

Indeed, we must have either $l=0$, or $k=0$, or $kl>0$. If $l=0$, then \eqref{eq:axnpbxkylpcym} reduces to $a x^n + b x^k + c y^m = 0$, hence \eqref{eq:Rungexnpxkylpym} can be ensured by swapping $(a,n)$ with $(b,k)$ if necessary. The case $k=0$ is similar. Finally, if $kl>0$ and $nl+mk>mn$, then substitution $X=x^d$, $Y=y^d$ where $d=\text{gcd}(n,m,k,l)$ reduces \eqref{eq:axnpbxkylpcym} to 
$$
a X^{n/d} + b X^{k/d} Y^{l/d} + c Y^{m/d} = 0, \quad \text{where} \quad \frac{n}{d} \frac{l}{d} + \frac{m}{d} \frac{k}{d} > \frac{m}{d} \frac{n}{d}.
$$
Because $\text{gcd}(n/d,m/d,k/d,l/d)=1$, polynomial $P(X,Y)=a X^{n/d} + b X^{k/d} Y^{l/d} + c Y^{m/d}$ is irreducible by \cite[Example 1]{MR1816701}, and condition (C1) in Theorem \ref{th:Runge} is satisfied. Hence, by Theorem \ref{th:Runge}, equation $P(X,Y)=0$ has finitely many integer solutions and there is an algorithm for listing these solutions. Then it is easy to check which of these solutions are perfect $d$-th powers, and produce a complete list of integer solutions $(x,y)$ to \eqref{eq:axnpbxkylpcym}.
%Finally, if $kl>0$ and $nl+mk>mn$ then \eqref{eq:axnpbxkylpcym} is trivial if $mn=0$, and otherwise \eqref{eq:axnpbxkylpcym} is irreducible and is covered by Theorem \ref{th:Runge}. 	
	
%We may assume that the equation is of the form \eqref{eq:axnpbxkylpcym} with non-zero integers $a,b,c$ and non-negative integers $n,k,l,m$ satisfying \eqref{eq:Rungexnpxkylpym}. 
From now on, we will consider equation \eqref{eq:axnpbxkylpcym} subject to condition \eqref{eq:Rungexnpxkylpym}. First assume that equality holds in \eqref{eq:Rungexnpxkylpym}, that is, $nl+mk=mn$, or $mk=n(m-l)$. By solving this as an equation of the form \eqref{eq:xypzt} with variables $m,k,n,m-l$, we conclude that there exists integers $u,v,w,r$ such that $(m,k,n,m-l)=(uv,wr,uw,vr)$, see \eqref{eq:xypztsol}. By replacing $u,v,w,r$ by their absolute values if necessary, we may assume that these integers are non-negative.    
Now divide both parts of \eqref{eq:axnpbxkylpcym} by $y^m$ to get
$$
a \frac{x^n}{y^{m}} + b \frac{x^k}{y^{m-l}} + c = 0,
$$
or equivalently
$$
a \left(\frac{x^w}{y^v}\right)^u + b \left(\frac{x^w}{y^v}\right)^r + c = 0.
$$
With a new rational variable $t=\frac{x^w}{y^v}$, this is an equation $at^u+bt^r+c=0$ in one variable, whose rational solutions are easy to find, see the discussion after equation \eqref{eq:onevar} in Section \ref{sec:2var}. For each rational solution $t=\frac{p_i}{q_i}$,  we can find $x$ and $y$ from an easy two-monomial equation $q_i x^w = p_i y^v$.

From now on, we will assume that
\begin{equation}\label{eq:Rungexnpxkylpym2}
nl+mk < mn.
\end{equation}

Let $(x,y)$ be any integer solution to \eqref{eq:axnpbxkylpcym}. Because solutions with $xy=0$ can be described easily, we may and will assume that $xy\neq 0$.
%To illustrate the technique, we first consider the case $a=b=c=1$ in \eqref{eq:axnpbxkylpcym}, that is, equation
%\begin{equation}\label{eq:xnpxkylpym}
%x^n + x^k y^l + y^m = 0.
%\end{equation}
Let $p$ be any prime, and let $x_p$ and $y_p$ be the exponents with which $p$ enters the prime factorizations of $x$ and $y$, respectively. First assume that $p$ is not a divisor of $abc$. Then $p$ enters the prime factorization of the monomials $a x^n$, $b x^k y^l$ and $c y^m$ in \eqref{eq:axnpbxkylpcym} with exponents $nx_p$, $kx_p+ly_p$ and $my_p$, respectively. Let $M=\min\{nx_p,kx_p+ly_p,my_p\}$, and $M'$ be the second-smallest among these integers. By Lemma \ref{le:main}, $M=M'$. If $kx_p+ly_p > M$, then we must have
$$
kx_p+ly_p > nx_p = my_p.
$$ 
The last equality implies that $x_p=um$, $y_p=un$ for some rational $u\geq 0$. But then
$$
k(um) + l(un) > n(um) \quad \Leftrightarrow \quad u(km + ln) > u(nm),
$$
which is a contradiction with \eqref{eq:Rungexnpxkylpym2}. Hence, $kx_p+ly_p=M$. Then, for every prime $p$, we have one of the following two cases: either 
\begin{equation}\label{eq:primecase1}
	kx_p+ly_p = nx_p \leq my_p
\end{equation}
or 
\begin{equation}\label{eq:primecase2}
	kx_p+ly_p = my_p < nx_p.
\end{equation}
In the first case, we have
$
ly_p = (n-k)x_p.
$
By \eqref{eq:Rungexnpxkylpym2}, $n-k>0$, hence, by \eqref{eq:nxmysol}, 
\begin{equation}\label{eq:primecondI}
x_p = \frac{l}{\text{gcd}(l,n-k)} u_p, \quad y_p = \frac{n-k}{\text{gcd}(l,n-k)} u_p
\end{equation}
for some integer $u_p \geq 0$ that depends on $p$. 
Similarly, in case \eqref{eq:primecase2}, we have
$
kx_p = (m-l)y_p,
$
which by \eqref{eq:nxmysol} implies that
\begin{equation}\label{eq:primecondII}
x_p = \frac{m-l}{\text{gcd}(k,m-l)} v_p, \quad y_p = \frac{k}{\text{gcd}(k,m-l)} v_p,
\end{equation}
for some integer $v_p \geq 0$. 
With notations
\begin{equation}\label{eq:lnmkprimedef}
l'=\frac{l}{\text{gcd}(l,n-k)}\,; \quad n'=\frac{n-k}{\text{gcd}(l,n-k)}\,; \quad m'=\frac{m-l}{\text{gcd}(k,m-l)}\,; \quad k'= \frac{k}{\text{gcd}(k,m-l)}\,,
\end{equation}
formulas \eqref{eq:primecondI} and \eqref{eq:primecondII} simplify to $x_p=l'u_p$, $y_p=n'u_p$ and $x_p=m'v_p$, $y_p=k'v_p$, respectively.

Now consider primes $p$ that are the factors of $abc$. Obviously, there are only finitely many such primes. Let $x_p,y_p,a_p,b_p,c_p$ be the exponents with which $p$ enters the prime factorizations of $x,y,a,b,c$, respectively. Then $p$ enters the prime factorization of the monomials $ax^n$, $bx^k y^l$ and $cy^m$ in \eqref{eq:axnpbxkylpcym} with exponents $a_p+nx_p$, $b_p+kx_p+ly_p$ and $c_p+my_p$, respectively. Let $M=\min\{a_p+nx_p,b_p+kx_p+ly_p,c_p+my_p\}$, and $M'$ be the second-smallest among these integers. By Lemma \ref{le:main}, $M=M'$.

First assume that $b_p+kx_p+ly_p > M$, then we must have
\begin{equation}\label{eq:subcase1}
b_p+kx_p+ly_p > a_p+nx_p = c_p+my_p.
\end{equation}
By Lemma \ref{le:nxmyb}, the last equality implies the existence of a pair of integers $(x_p^0,y_p^0)$ depending only on $p,m,n,a_p$ and $c_p$ such that
\begin{equation}\label{eq:firstcase}
	x_p = x^0_p + \frac{m}{\gcd(n,m)}u_p, \quad y_p = y^0_p + \frac{n}{\gcd(n,m)}u_p \quad \text{for some integer} \quad u_p \geq 0.
\end{equation}
But then the inequality in \eqref{eq:subcase1} 
$$
b_p + k\left(x^0_p + \frac{m}{\gcd(n,m)}u_p\right) + l\left(y^0_p + \frac{n}{\gcd(n,m)}u_p\right) > a_p + n\left(x^0_p + \frac{m}{\gcd(n,m)}u_p\right)
$$ 
is equivalent to
$$
b_p + k x^0_p + l y^0_p - a_p - n x^0_p > \frac{nm-km-ln}{\gcd(n,m)}u_p
$$
Let $h_p=b_p + k x^0_p + l y^0_p - a_p - n x^0_p$ be the left-hand side of the last inequality. Notice that $h_p$ does not depend on $x_p,y_p$. Because $nm-km-ln>0$ by \eqref{eq:Rungexnpxkylpym2}, we have 
$$
0 \leq u_p < h_p \frac{\gcd(n,m)}{nm-km-ln}.
$$ 
Because every bounded interval contains only a finite number of integers, there is only a finite number of possible integer values for $u_p$, hence a finite number of possibilities for $(x_p,y_p)$. 
%Let us denote this finite set as ${\cal F}_p$.

Now consider the case $b_p+kx_p+ly_p=M$. Because $M=M'$, we must have either 
\begin{equation}\label{eq:subcase2}
	b_p+kx_p+ly_p = a_p+nx_p \leq c_p+my_p
\end{equation}
or 
\begin{equation}\label{eq:subcase3}
	b_p+kx_p+ly_p = c_p+my_p < a_p+nx_p. 
\end{equation}
In the first case, 
$
b_p + l y_p = a_p + (n-k)x_p,
$
and, by Lemma \ref{le:nxmyb}, this implies the existence of pair of integers $(x_p^0,y_p^0)$ depending only on $p,l,n-k,a_p$ and $b_p$ such that
\begin{equation}\label{eq:primecondB}
	x_p = x^0_p + l' u_p, \quad y_p = y^0_p + n' u_p \quad \text{for some integer} \quad u_p \geq 0,
\end{equation}
where $l'$ and $n'$ are defined in \eqref{eq:lnmkprimedef}. 
Similarly, in the second case, there is a pair of integers $(x_p^0,y_p^0)$ depending only on $p,k,m-l,b_p$ and $c_p$ such that
\begin{equation}\label{eq:primecondC}
	x_p = x^0_p + m' v_p, \quad y_p = y^0_p + k' v_p \quad \text{for some integer} \quad v_p \geq 0,
\end{equation}
where $m'$ and $k'$ are defined in \eqref{eq:lnmkprimedef}.

The argument above implies that, given any non-zero integer solution $(x,y)$ to equation \eqref{eq:axnpbxkylpcym} satisfying \eqref{eq:Rungexnpxkylpym2}, we can partition the set ${\cal P}$ of all primes into five disjoint sets ${\cal P}=\bigcup_{i=1}^5 {\cal P}_i$, where 
\begin{itemize}
	\item ${\cal P}_1$ is the set of primes that are not divisors of $abc$ and satisfy \eqref{eq:primecase1};
	\item ${\cal P}_2$ is the set of primes that are not divisors of $abc$ and satisfy \eqref{eq:primecase2};
	\item ${\cal P}_3$ is the set of prime divisors of $abc$ satisfying \eqref{eq:subcase1};
	\item ${\cal P}_4$ is the set of prime divisors of $abc$ satisfying \eqref{eq:subcase2};
	\item ${\cal P}_5$ is the set of prime divisors of $abc$ satisfying \eqref{eq:subcase3}.
\end{itemize}

Let us define
\begin{equation}\label{eq:uvdef}
	u = \prod_{p \in {\cal P}_1} p^{u_p} \cdot \prod_{p \in {\cal P}_4} p^{u_p},  \quad v = \prod_{p \in {\cal P}_2} p^{v_p} \cdot \prod_{p \in {\cal P}_5} p^{v_p}.
\end{equation}
Then
$$
\prod_{p \in {\cal P}_1} p^{x_p} \cdot \prod_{p \in {\cal P}_4} p^{x_p} = \prod_{p \in {\cal P}_1} p^{l' u_p} \cdot \prod_{p \in {\cal P}_4} p^{x^0_p} \cdot \prod_{p \in {\cal P}_4} p^{l' u_p} = \left(\prod_{p \in {\cal P}_4} p^{x^0_p}\right)\cdot u^{l'},
$$
and similarly,
$$
\prod_{p \in {\cal P}_2} p^{x_p} \cdot \prod_{p \in {\cal P}_5} p^{x_p} = \left(\prod_{p \in {\cal P}_5} p^{x^0_p}\right)\cdot v^{m'}.
$$
Hence,
$$
x = (-1)^{e_x} \prod_{p \in {\cal P}_1} p^{x_p} \cdot \prod_{p \in {\cal P}_2} p^{x_p} \prod_{p \in {\cal P}_3} p^{x_p} \cdot \prod_{p \in {\cal P}_4} p^{x_p} \prod_{p \in {\cal P}_5} p^{x_p} =
% = \prod_{p \in {\cal P}_1} p^{l' u_p} \cdot \prod_{p \in {\cal P}_2} p^{m' u_p} = 
$$
$$
= \left((-1)^{e_x} \prod_{p \in {\cal P}_3} p^{x_p} \prod_{p \in {\cal P}_4} p^{x^0_p} \prod_{p \in {\cal P}_5} p^{x^0_p}\right) u^{l'} v^{m'}.  
$$
where $e_x$ is equal to either $0$ or $1$ depending on the sign of $x$.
By similar argument,
$$
y = \left((-1)^{e_y} \prod_{p \in {\cal P}_3} p^{y_p} \prod_{p \in {\cal P}_4} p^{y^0_p} \prod_{p \in {\cal P}_5} p^{y^0_p}\right) u^{n'} v^{k'},
$$
where $e_y$ is equal to either $0$ or $1$. 

Now, for each given equation \eqref{eq:axnpbxkylpcym}, there is only a finite number of prime divisors of $abc$, hence there is a finite number of possibilities as to how these primes may split into sets ${\cal P}_3$, ${\cal P}_4$ and ${\cal P}_5$. For each given split, integers $\prod_{p \in {\cal P}_4} p^{x^0_p}$, $\prod_{p \in {\cal P}_5} p^{x^0_p}$, $\prod_{p \in {\cal P}_4} p^{y^0_p}$ and $\prod_{p \in {\cal P}_5} p^{y^0_p}$ are uniquely determined. Further, for every $p \in {\cal P}_3$ there is only a finite number of possibilities for $(x_p,y_p)$, hence we have a finite number of possible values for $\prod_{p \in {\cal P}_3} p^{x_p}$ and $\prod_{p \in {\cal P}_3} p^{y_p}$. Finally, $e_x$ and $e_y$ may assume only two values each. In conclusion, we have proved 
%$$
%\left(\prod_{p \in {\cal P}} p^{u_p}\right)^{l'} \cdot \left(\prod_{p \not \in {\cal P}} p^{u_p}\right)^{m'} = u^{l'} v^{m'},
%$$
%There is a finite number of prime divisors of $abc$, and for each such prime $p$ we have a finite number of possibilities for $u_p$ in \eqref{eq:firstcase}, and for $(x^0_p,y^0_p)$ in \eqref{eq:primecondB} and \eqref{eq:primecondC}. This implies that there 
that there is a finite number of pairs $(x_i,y_i)$, $i=1,\dots,N$ such that every solution to \eqref{eq:axnpbxkylpcym} with $xy\neq 0$ satisfies 
$$
x = x_i u^{l'} v^{m'}, \quad y = y_i u^{n'} v^{k'}
$$ 
for some $i=1,\dots, N$ and some integers $u$ and $v$. 
Substituting this into \eqref{eq:axnpbxkylpcym}, we obtain
$$
a x_i^n u^{nl'} v^{nm'} + b x_i^k y_i^l u^{kl'} v^{km'} u^{ln'} v^{lk'} + c y_i^m u^{mn'} v^{mk'} = 0,
$$
or, after dividing by $u^{kl'+ln'}v^{km'+lk'}$,
\begin{equation}\label{eq:xnpxkylpymred2}
	a x_i^n u^{nl'-kl'-ln'}v^{nm'-km'-lk'} + b x_i^k y_i^l + c y_i^m u^{mn'-kl'-ln'}v^{mk'-km'-lk'} = 0.
\end{equation}
Now, we have
\begin{equation}\label{eq:cond1}
	nl'-kl'-ln' = (n-k)\frac{l}{\text{gcd}(l,n-k)}-l\frac{n-k}{\text{gcd}(l,n-k)} = 0,
\end{equation}
\begin{equation}\label{eq:cond2}
	nm'-km'-lk' = (n-k)\frac{m-l}{\text{gcd}(k,m-l)}-l\frac{k}{\text{gcd}(k,m-l)} = \frac{nm-nl-km}{\text{gcd}(k,m-l)} > 0,
\end{equation}
where the last inequality follows from \eqref{eq:Rungexnpxkylpym2}. Similarly, $mn'-kl'-ln'> 0$ and $mk'-km'-lk'= 0$. Hence, \eqref{eq:xnpxkylpymred2} reduces to
\begin{equation}\label{eq:xnpxkylpymred}
	a x_i^n v^{nm'-km'-lk'} + b x_i^k y_i^l + c y_i^m u^{mn'-kl'-ln'} = 0,
\end{equation}
which is an equation of the form \eqref{eq:cympaxnpb} in variables $u,v$. In conclusion, we have reduced equation  \eqref{eq:axnpbxkylpcym} to a finite number of equations of the form \eqref{eq:cympaxnpb} that can be solved by Theorem \ref{th:ellBaker1}. 
\end{proof}

Let us now illustrate the method by solving equation  
\begin{equation}\label{eq:x4pxyp2y3}
	x^4+xy+2y^3=0. 
\end{equation}
We have an obvious solution $(x,y)=(0,0)$ and otherwise $x\neq 0$ and $y\neq 0$. For this equation, $a=b=1$ and $c=2$ in \eqref{eq:axnpbxkylpcym}, hence the only prime divisor of $abc$ is $p=2$. For any $p\neq 2$, let $x_p$ and $y_p$ be the exponents with which $p$ enters the prime factorizations of $x$ and $y$, respectively. Then $p$ enters the prime factorization of the monomials $x^4$, $xy$ and $2y^3$ with exponents $4x_p$, $x_p+y_p$ and $3y_p$, respectively. Let $M=\min\{4x_p,x_p+y_p,3y_p\}$, and $M'$ be the second-smallest among these integers. Then, by Lemma \ref{le:main}, we must have $M=M'$. Because the case $x_p+y_p > 4x_p = 3y_p$ is clearly impossible for non-negative $x_p,y_p$, we must have either $x_p+y_p =4x_p$ or $x_p+y_p=3y_p$, or, equivalently, either $y_p=3x_p$ or $x_p=2y_p$. Let ${\cal P}_1$ and ${\cal P}_2$ be the sets of odd primes for which the first and the second equalities hold, respectively, and define non-zero integers
$$
u = (-1)^{e_x}\prod_{p \in {\cal P}_1} p^{x_p}, \quad v = (-1)^{e_x+e_y}\prod_{p \in {\cal P}_2} p^{y_p},
$$
where $e_x=0$ if $x>0$ and $e_x=1$ if $x<0$, and $e_y$ is defined similarly. Then
$$
x = (-1)^{e_x}2^{x_2} \prod_{p \in {\cal P}_1} p^{x_p} \prod_{p \in {\cal P}_2} p^{x_p} = (-1)^{e_x} 2^{x_2} \prod_{p \in {\cal P}_1} p^{x_p} \prod_{p \in {\cal P}_2} p^{2 y_p} = 2^{x_2} u v^2, 
$$
and similarly
$$
y = (-1)^{e_y}2^{y_2} \prod_{p \in {\cal P}_1} p^{y_p} \prod_{p \in {\cal P}_2} p^{y_p} = (-1)^{e_y} 2^{y_2} \prod_{p \in {\cal P}_1} p^{3 x_p} \prod_{p \in {\cal P}_2} p^{y_p} = 2^{y_2} u^3 v, 
$$
where $x_2$ and $y_2$ are the exponents with which $2$ enters the prime factorizations of $x$ and $y$, respectively. Then $2$ enters the prime factorization of the monomials $x^4$, $xy$ and $2y^3$ with exponents $4x_2$, $x_2+y_2$ and $3y_2+1$, respectively. Let $M_1=\min\{4x_2,x_2+y_2,3y_2+1\}$, and $M_1'$ be the second-smallest among these integers. Then we must have $M_1=M_1'$. The case $x_2+y_2 \geq 4x_2 = 3y_2+1$ is clearly impossible for non-negative $x_2,y_2$, hence we must have either (i) $x_2+y_2 = 4x_2$ or (ii) $x_2+y_2 = 3y_2 +1 $. 

\begin{table}
	\footnotesize
	\begin{center}
		\begin{tabular}{ | c | c |c|c|c|c|}
			\hline
			\textbf{Equation \eqref{eq:xnpxkylpym}} & $|x|$ & $|y|$ & \textbf{Equation \eqref{eq:xnpxkylpymred}}& \textbf{Solutions $(u,v)$}&\textbf{Solutions $(x,y) \neq (0,0)$}\\
			\hline\hline

			$x^4+x y+y^2=0$ & $u v$ & $u^3 v$ & $\pm 1 + u^2 + v^2=0$ &$(0,\pm 1),(\pm 1,0)$&-\\\hline 
			$x^4+x y+y^3=0$ & $u v^2$ & $u^3 v$ & $1+u^5+v^5=0$&$(-1,0),(0,-1)$&- \\\hline 
			$x^4+x y+y^4=0$ & $u v^3$ & $u^3 v$ & $\pm 1 + u^8 + v^8=0$&$(0,\pm 1),(\pm1,0)$&- \\\hline 
			$x^4+x y^2+y^3=0$ & $u^2 v$ & $u^3 v$ & $1+u+v=0$&$(w,-1-w)$&$(-w^2(1+w),-w^3(1+w))$ \\\hline 
			$x^4+x^2 y+y^2=0$ & $u v$ & $u^2 v^2$ & $1=0$ &-&-\\\hline 
			$x^4+x^2 y+y^3=0$ & $u v$ & $u^2 v$ & $1+ u^2+v=0$ &$(w,-1-w^2)$&$(-w-w^3,-w^2(1+w^2))$\\\hline 
			$x^4+x^2 y+y^4=0$ & $u v^3$ & $u^2 v^2$ & $\pm 1+ u^4 + v^4=0$&$(0,\pm 1),(\pm1,0)$&- \\\hline 
			$x^4+x^2 y^2+y^4=0$ & $u v$ & $u v$ & $1=0$ &-&- \\\hline 
			$x^4+x^3 y+y^4=0$ & $u v$ & $u v$ & $1=0$ &- &-\\\hline 
			$x^5+x y+y^2=0$ & $u v$ & $u^4 v$ & $1+u^3+v^3=0$&$(-1,0),(0,-1)$&- \\\hline 
			$x^5+x y+y^3=0$ & $u v^2$ & $u^4 v$ & $1+u^7+v^7=0$ &$(-1,0),(0,-1)$&-\\\hline 
			$x^5+x y+y^4=0$ & $u v^3$ & $u^4 v$ & $1+u^{11}+v^{11}=0$ &$(-1,0),(0,-1)$&- \\\hline 
			$x^5+x y+y^5=0$ & $u v^4$ & $u^4 v$ & $1+u^{15}+v^{15}=0$  &$(-1,0),(0,-1)$&- \\\hline
			$x^5+x y^2+y^3=0$ & $u v$ & $u^2 v$ & $1+u+ v^2=0$&$(-1-w^2,w)$&$(-w-w^3,w(1+w^2)^2)$ \\\hline 
			$x^5+x y^2+y^4=0$ & $u v^2$ & $u^2 v$ & $1+u^3 + v^6=0$&$(-1,0)$&- \\\hline 
			$x^5+x y^3+y^4=0$ & $u^3 v$ & $u^4 v$ & $1+u+v=0$&$(w,-1-w)$&$(-w^3(1+w),-w^4(1+w))$ \\\hline 
			$x^5+x^2 y+y^2=0$ & $u v$ & $u^3 v^2$ & $1+u+v=0$&$(w,-1-w)$&$(-w-w^2,w^3(1+w)^2)$ \\\hline 
			$x^5+x^2 y+y^3=0$ & $u v$ & $u^3 v$ & $1 + u^4 \pm v^2=0$&$(0,\pm 1)$&- \\\hline 
			$x^5+x^2 y+y^4=0$ & $u v^3$ & $u^3 v^2$ & $1+u^7+v^7=0$&$(-1,0),(0,-1)$&- \\\hline 
			$x^5+x^2 y+y^5=0$ & $u v^2$ & $u^3 v$ & $1 + u^{10} + v^5=0$&$(0,-1)$&- \\\hline 
			$x^5+x^2 y^2+y^4=0$ & $u^2 v$ & $u^3 v$ & $1+ u^2+v=0$&$(w,-1-w^2)$&$(-w^2(1+w^2),-w^3(1+w^2))$ \\\hline 
			$x^5+x^2 y^2+y^5=0$ & $u^2 v^3$ & $u^3 v^2$ & $1+u^5+v^5=0$&$(-1,0),(0,-1)$&- \\\hline
			$x^5+x^3 y+y^3=0$ & $u v^2$ & $u^2 v^3$ & $1+u+v=0$&$(w,-1-w)$&$(w(1+w)^2,-w^2(1+w)^3)$ \\\hline
			$x^5+x^3 y+y^4=0$ & $u v$ & $u^2 v$ & $1+u^3+v=0$&$(w,-1-w^3)$&$(-w-w^4,-w^2(1+w^3))$ \\\hline 
			$x^5+x^3 y+y^5=0$ & $u v^4$ & $u^2 v^3$ & $1+u^5+v^5=0$&$(-1,0),(0,-1)$&- \\\hline 
			$x^5+x^3 y^2+y^5=0$ & $u v$ & $u v$ & $1=0$ &- &-\\\hline 
			$x^5+x^4 y+y^5=0$ & $u v$ & $u v$ & $1=0$ &-&- \\\hline 
		\end{tabular}
		\caption{Non-trivial solutions to equations \eqref{eq:xnpxkylpym} of degree $4$ and $5$ satisfying  \eqref{eq:Rungexnpxkylpym}. Here, $w$ is an arbitrary integer parameter. \label{tab:xnpxkylpym}}
	\end{center} 
\end{table}

Let us consider these cases separately. In case (i), $y_2=3x_2$, so that
$$
x = 2^{x_2} u v^2 = Uv^2, \quad 
y=2^{y_2} u^3 v = (2^{x_2})^3 u^3 v = U^3 v, \quad \text{where} \quad U=2^{x_2} u \neq 0.
$$ 
Substituting this into equation \eqref{eq:x4pxyp2y3}, we obtain
$$
(Uv^2)^4+(Uv^2)(U^3 v)+2(U^3 v)^3=0,
$$
or after cancelling $U^4v^3$, 
$$
v^5 + 2U^5 = -1.
$$
This is a Thue equation, and such equations can be easily solved in many computer algebra systems including Maple and Mathematica. The solutions to this equation are $(v,U)=(-1,0)$ and $(1,-1)$. The first solution is impossible because $U\neq 0$ by definition, while the second one results in the solution $(x,y)=(Uv^2,U^3v)=(-1,-1)$ to the original equation \eqref{eq:x4pxyp2y3}. 

In case (ii), $x_2= 2y_2+1$, hence
$$
x = 2^{x_2} u v^2 =  2(2^{y_2})^2 u v^2 = 2uV^2, \quad 
y = 2^{y_2} u^3 v = u^3 V, \quad \text{where} \quad V=2^{y_2} v \neq 0.
$$ 
Substituting this into equation \eqref{eq:x4pxyp2y3}, we obtain
$$
(2uV^2)^4+(2uV^2)(u^3 V)+2(u^3 V)^3=0,
$$
or after cancelling $2u^4V^3$, 
$$
8V^5 + u^5 = -1.
$$
This is a Thue equation, whose only solution is $(V,u)=(0,-1)$, a contradiction with $V\neq 0$. In summary, the only integer solutions to equation \eqref{eq:x4pxyp2y3} are $(x,y)=(-1,-1)$ and $(0,0)$.

The algorithm in the proof of Theorem \ref{th:3mon2var} significantly simplifies in the case $a=b=c=1$ in \eqref{eq:axnpbxkylpcym}, that is, for equation
\begin{equation}\label{eq:xnpxkylpym}
	x^n + x^k y^l + y^m = 0,
\end{equation}
because in this case, the sets ${\cal P}_3$, ${\cal P}_4$ and ${\cal P}_5$ in the algorithm are empty sets. As an illustration, Table \ref{tab:xnpxkylpym} lists all non-trivial solutions to all quartic and quintic equations of the form \eqref{eq:xnpxkylpym} satisfying the condition \eqref{eq:Rungexnpxkylpym}.

\section{General three-monomial equations}\label{sec:gen3mon}

\subsection{Format in which we accept the answers}\label{sec:solrepr}

When studying three-monomial equations in $n\geq 3$ variables, we first need to agree in what form we accept the answers. For example, if we insist on polynomial parametrization of all integer solutions, then the equation
\begin{equation}\label{eq:xymztm1}
xy-zt=1
\end{equation}
had been open for over $70$ years, until Vaserstein \cite{vaserstein2010polynomial} proved the existence of polynomials $X,Y,Z,T$ in $46$ variables $u=(u_1,\dots,u_{46})$ with integer coefficients, such that $(x,y,z,t)\in {\mathbb Z}^4$ is a solution to \eqref{eq:xymztm1} if and only if $x=X(u)$, $y=Y(u)$, $z=Z(u)$ and $t=T(u)$ for some $u \in {\mathbb Z}^{46}$. We do not know if there exist polynomial parametrizations of all integer solutions to other easy-looking three-monomial equations, like, for example, 
\begin{equation}\label{eq:x2ymz2m1}
x^2y=z^2+1
\end{equation} 
or
\begin{equation}\label{eq:yztmx2m1}
yzt=x^2+1.
\end{equation}  
On the other hand, all these equations become trivial if we accept answers in which parameters can be restricted to be divisors of polynomial or rational expressions involving other parameters. Specifically, for any integers $m$ and $k\geq 0$, let $D_k(m)$ be the set of all integers $z$ such that $z^k$ is a divisor of $m$. We remark that $D_0(m)$ is the set of all integers, while $D_1(m)$ is the set of all divisors of $m$. This notation allows us to describe the sets of solutions to equations \eqref{eq:xymztm1}-\eqref{eq:yztmx2m1}. For example, the set of all integer solutions to equation \eqref{eq:yztmx2m1} can be described as
$$
(x,y,z,t) = \left(u_1, u_2, u_3, \frac{u_1^2+1}{u_2u_3}\right), \quad u_1 \in {\mathbb Z}, \quad u_2 \in D_1\left(u_1^2+1\right), \quad u_3 \in D_1\left(\frac{u_1^2+1}{u_2}\right).
$$

More generally, we can easily describe the solution sets of all equations that, possibly after permutation of variables, can be written in the form
\begin{equation}\label{eq:x1kx2}
	P(x_1^kx_2, x_3, \dots, x_n) = 0,
\end{equation}
where $k\geq 1$ is an integer and $P$ is a polynomial with integer coefficients such that we know how to describe the integer solution set $S$ of the equation $P(y,x_2,\dots,x_n)=0$. Indeed, integer solutions to \eqref{eq:x1kx2} with $x_1\neq 0$ are
$$
(x_1, x_2, \dots, x_n) = \left(u_1, \frac{u_2}{u_1^k}, u_3, \dots, u_n \right), \quad (u_2,u_3,\dots,u_n) \in S, \quad u_1 \in D_k(u_2),
$$
while the solutions with $x_1=0$ are such that $(0,x_3,\dots,x_n)\in S$ and $x_2$ arbitrary. We remark that equations \eqref{eq:xymztm1}-\eqref{eq:yztmx2m1} are all of the form \eqref{eq:x1kx2}.

\subsection{A direct formula for solving a large class of three-monomial equations}\label{sec:dirfor}

A general three-monomial Diophantine equation can be written in the form
\begin{equation}\label{eq:gen3mon2}
	a \prod_{i=1}^n x_i^{\alpha_i} + b \prod_{i=1}^n x_i^{\beta_i} = c \prod_{i=1}^n x_i^{\gamma_i},
\end{equation} 
where $x_1,\dots,x_n$ are variables, $\alpha_i$, $\beta_i$ and $\gamma_i$ are non-negative integers, and $a,b,c$ are non-zero integer coefficients. We call a solution $x=(x_1,\dots,x_n)$ to \eqref{eq:gen3mon2} trivial if $\prod_{i=1}^n x_i = 0$ and non-trivial otherwise. The trivial solutions are easy to describe, so we may concentrate on finding the non-trivial ones.

We start by identifying a large class of three-monomial equations whose set of non-trivial integer solutions can be described by an easy and direct formula. 

\begin{proposition}\label{prop:3monformula}
	Assume that a Diophantine equation can be written in the form \eqref{eq:gen3mon2} such that both systems 
	\begin{equation}\label{eq:gen3monsyst1}
		\sum_{i=1}^n \alpha_i z_i = \sum_{i=1}^n \beta_i z_i = \sum_{i=1}^n \gamma_i z_i - 1, \quad \quad z_i \geq 0, \quad i=1,\dots, n, 
	\end{equation}
	and
	\begin{equation}\label{eq:gen3monsyst2}
		\sum_{i=1}^n \alpha_i t_i = \sum_{i=1}^n \beta_i t_i = \sum_{i=1}^n \gamma_i t_i + 1, \quad \quad t_i \geq 0, \quad i=1,\dots, n,
	\end{equation}
    are solvable in non-negative integers $z_i$ and $t_i$. Then all non-trivial integer solutions to \eqref{eq:gen3mon2} are given by 
	\begin{equation}\label{eq:gen3mon2sol}
		x_i = \left(a \prod_{j=1}^n u_j^{\alpha_j} + b \prod_{j=1}^n u_j^{\beta_j}\right)^{z_i} \left(c \prod_{j=1}^n u_j^{\gamma_j}\right)^{t_i} w^{-z_i-t_i} \cdot u_i, \quad i=1,\dots, n, 
	\end{equation} 
	where
	$$
	u_i \in {\mathbb Z}, \, i=1,\dots, n, \quad 
	w \in D_1\left(a \prod_{j=1}^n u_j^{\alpha_j} + b \prod_{j=1}^n u_j^{\beta_j}\right) \cap D_1\left(c \prod_{j=1}^n u_j^{\gamma_j}\right).
	%w \in \bigcap_{i=1}^n D_{z_i+t_i}\left(\left(a \prod_{i=1}^n u_i^{\alpha_i} + b \prod_{i=1}^n u_i^{\beta_i}\right)^{z_i} \left(c \prod_{i=1}^n u_i^{\gamma_i}\right)^{t_i} \right).
	$$
\end{proposition}
\begin{proof}
Let us check that any $(x_1,\dots, x_n)$ satisfying \eqref{eq:gen3mon2sol} is an integer solution to \eqref{eq:gen3mon2}. The condition on $w$ ensures that all $x_i$ in \eqref{eq:gen3mon2sol} are integers. 
%For any $(x_1,\dots, x_n)$ satisfying \eqref{eq:gen3mon2sol}, 
Further, 
$$
a \prod_{i=1}^n x_i^{\alpha_i} = \left(a \prod_{i=1}^n u_i^{\alpha_i} + b \prod_{i=1}^n u_i^{\beta_i}\right)^{\sum_{i=1}^n \alpha_i z_i} \left(c \prod_{i=1}^n u_i^{\gamma_i}\right)^{\sum_{i=1}^n \alpha_i t_i} w^{-\sum_{i=1}^n \alpha_i (z_i+t_i)} \cdot a \prod_{i=1}^n u_i^{\alpha_i},
$$
and
$$
b \prod_{i=1}^n x_i^{\beta_i} = \left(a \prod_{i=1}^n u_i^{\alpha_i} + b \prod_{i=1}^n u_i^{\beta_i}\right)^{\sum_{i=1}^n \beta_i z_i} \left(c \prod_{i=1}^n u_i^{\gamma_i}\right)^{\sum_{i=1}^n \beta_i t_i} w^{-\sum_{i=1}^n \beta_i (z_i+t_i)} \cdot b \prod_{i=1}^n u_i^{\beta_i},
$$
The first equalities in \eqref{eq:gen3monsyst1} and \eqref{eq:gen3monsyst2} imply that the first three factors in these expressions are the same, hence
$$
a \prod_{i=1}^n x_i^{\alpha_i} + b \prod_{i=1}^n x_i^{\beta_i} = \left(a \prod_{i=1}^n u_i^{\alpha_i} + b \prod_{i=1}^n u_i^{\beta_i}\right)^{1+\sum_{i=1}^n \alpha_i z_i} \left(c \prod_{i=1}^n u_i^{\gamma_i}\right)^{\sum_{i=1}^n \alpha_i t_i} w^{-\sum_{i=1}^n \alpha_i (z_i+t_i)}.
$$
On the other hand
$$
c \prod_{i=1}^n x_i^{\gamma_i} = \left(a \prod_{i=1}^n u_i^{\alpha_i} + b \prod_{i=1}^n u_i^{\beta_i}\right)^{\sum_{i=1}^n \gamma_i z_i} \left(c \prod_{i=1}^n u_i^{\gamma_i}\right)^{\sum_{i=1}^n \gamma_i t_i} w^{-\sum_{i=1}^n \gamma_i (z_i+t_i)} \cdot c \prod_{i=1}^n u_i^{\gamma_i}.
$$
The second equalities in \eqref{eq:gen3monsyst1} and \eqref{eq:gen3monsyst2} imply 
that the right-hand sides of the last two expressions are the same, and \eqref{eq:gen3mon2} follows. 

Conversely, let us prove that all non-trivial integer solutions to \eqref{eq:gen3mon2} are covered by \eqref{eq:gen3mon2sol} for some values of the parameters. Indeed, for any given solution $(x_1, \dots, x_n)$ let us take $u_i=x_i$, $i=1,\dots, n$, and $w=c \prod_{i=1}^n x_i^{\gamma_i}$. Then the right-hand side of \eqref{eq:gen3mon2sol} becomes
$$
\left(a \prod_{j=1}^n x_j^{\alpha_j} + b \prod_{j=1}^n x_j^{\beta_j}\right)^{z_i} \left(c \prod_{j=1}^n x_j^{\gamma_j}\right)^{t_i} w^{-z_i-t_i} \cdot x_i = \frac{\left(c \prod_{j=1}^n x_j^{\gamma_j}\right)^{z_i} \left(c \prod_{j=1}^n x_j^{\gamma_j}\right)^{t_i}}{w^{z_i+t_i}} \cdot x_i = x_i,
$$
and \eqref{eq:gen3mon2sol} follows.
\end{proof}

We remark that the same method can also be used to solve the general two-monomial equation, that is, the equation of the form
\begin{equation}\label{eq:gen2mon}
	a \prod_{i=1}^n x_i^{\alpha_i} = b \prod_{i=1}^n x_i^{\gamma_i},
\end{equation} 
where $x_1,\dots,x_n$ are variables, $\alpha_i$ and $\gamma_i$ are non-negative integers, and $a,b$ are non-zero integer coefficients. As above, we will look at non-trivial solutions, that is, ones with $\prod_{i=1}^n x_i \neq 0$. We may assume that $\min\{\alpha_i,\gamma_i\}=0$ for all $i$, otherwise the multiplier $x_i^{\min\{\alpha_i,\gamma_i\}}$ can be cancelled out. Equation \eqref{eq:gen2mon} can be rewritten as 
\begin{equation}\label{eq:gen2mon2}
	\prod_{i=1}^n x_i^{e_i} = r,
\end{equation}
where $e_i=\alpha_i-\gamma_i$, $i=1,\dots,n$ are integers and $r=\frac{b}{a}$ is a rational number. If $e_i\geq 0$ for all $i$, then \eqref{eq:gen2mon2} may have integer solutions only if $r$ is an integer. Because $x_i^{e_i}$ are divisors of $r$ for all $i$, all solutions can be easily found. The case $e_i\leq 0$ for all $i$ is similarly easy, so from now on we assume that not all $e_i$ are of the same sign. If $d=\text{gcd}(e_1,\dots,e_n)$, then \eqref{eq:gen2mon2} may be solvable in integers only if $\sqrt[d]{r}$ is a rational number, say $\frac{B}{A}$. In this case, raising both parts of \eqref{eq:gen2mon} to the power $1/d$ results in\footnote{If $d$ is even, then $L=R$ is equivalent to $L^{1/d}=\pm R^{1/d}$, hence \eqref{eq:gen2mon} reduces to two equations of the form \eqref{eq:gen2monred}.}
\begin{equation}\label{eq:gen2monred}
	A \prod_{i=1}^n x_i^{\alpha'_i} = B \prod_{i=1}^n x_i^{\gamma'_i},
\end{equation}  
where $\alpha'_i=\alpha_i/d$ and $\gamma'_i=\gamma_i/d$, $i=1,\dots, n$ are non-negative integers. The two-monomial versions of systems \eqref{eq:gen3monsyst1} and \eqref{eq:gen3monsyst2} are
$$
	\sum_{i=1}^n \alpha'_i z_i = \sum_{i=1}^n \gamma'_i z_i - 1, \quad \quad z_i \geq 0, \quad i=1,\dots, n, 
$$
and
$$
	\sum_{i=1}^n \alpha'_i t_i = \sum_{i=1}^n \gamma'_i t_i + 1, \quad \quad t_i \geq 0, \quad i=1,\dots, n,
$$
respectively. Because integers $\alpha'_i - \gamma'_i$, $i=1,\dots, n$ are not all of the same sign and do not have a common factor, both these systems are solvable in non-negative integers. Then, by analogy with \eqref{eq:gen3mon2sol}, all non-trivial solutions to \eqref{eq:gen2monred} are given by	\begin{equation}\label{eq:gen2monredsol}
	x_i = \left(A \prod_{j=1}^n u_j^{\alpha'_j}\right)^{z_i} \left(B \prod_{j=1}^n u_j^{\gamma'_j}\right)^{t_i} w^{-z_i-t_i} \cdot u_i, \quad i=1,\dots, n, 
\end{equation} 
where
$$
u_i \in {\mathbb Z}, \, i=1,\dots, n, \quad 
w \in D_1\left(A \prod_{j=1}^n u_j^{\alpha'_j}\right) \cap D_1\left(B \prod_{j=1}^n u_j^{\gamma'_j}\right).
$$

In particular, we have the following result, which we will use later.
\begin{proposition}\label{prop:2monex}
	Let $e_1, \dots, e_k$ be non-zero integers, not all of the same sign, and let $r$ be a rational number. Then equation \eqref{eq:gen2mon2} is solvable in integers $x_1, \dots, x_k$ if and only if $\sqrt[d]{r}$ is a rational number, where $d=\text{gcd}(e_1, \dots, e_k)$.
\end{proposition}

\subsection{Reduction to equations with independent monomials}

In this section we present a general method that reduces any three-monomial equation to a finite number of three-monomial equations with independent monomials, that is, equations of the form
\begin{equation}\label{eq:genindmon}
	a \prod_{i=1}^{n_1} x_i^{\alpha_i} + b \prod_{i=1}^{n_2} y_i^{\beta_i} + c \prod_{i=1}^{n_3} z_i^{\gamma_i} = 0,
\end{equation} 
where $a,b,c$ are integers, $n_1,n_2,n_3$, $\alpha_i$, $\beta_i$ and $\gamma_i$ are non-negative integers, and $x_i,y_i$ and $z_i$ are (different) variables. We remark that if an equation with independent monomials has a  monomial of degree $1$, then it is an equation of the form
\begin{equation}\label{eq:sepvara}
	a x_i + P(x_1,\dots,x_{i-1},x_{i+1},\dots,x_n) = 0, \quad a\neq 0,
\end{equation}
which is trivially solvable. Indeed, if $|a|=1$, then we can take $x_1,\dots,x_{i-1},x_{i+1},\dots,x_n$ to be arbitrary integers and express $x_i$ as $x_i=-P(x_1,\dots,x_{i-1},x_{i+1},\dots,x_n)/a$. If $|a|\geq 2$, \eqref{eq:sepvara} implies that $P(x_1,\dots,x_{i-1},x_{i+1},\dots,x_n)$ must be divisible by $|a|$. So, let us solve the equation $P(x_1,\dots,x_{i-1},x_{i+1},\dots,x_n)=0$ modulo $|a|$, and for each of the (finitely many) solutions $r_1,\dots,r_{i-1},r_{i+1},\dots,r_n$ do the substitution $x_j = |a|y_j + r_j$, $j=1,\dots,i-1,i+1,\dots,n$, where $y_j$ are new integer variables. Then \eqref{eq:sepvara} reduces to 
$$
a x_i = -P(|a|y_1+r_1,\dots,|a|y_{i-1}+r_{i-1},|a|y_{i+1}+r_{i+1},\dots,|a|y_n+r_n),
$$
where all the coefficients of the polynomial on the right-hand side are divisible by $|a|$. We can then cancel $|a|$ and reduce each of these equations to an equation of the form \eqref{eq:sepvara} with $|a|=1$.

An integer solution to \eqref{eq:genindmon} will be called \emph{primitive} if 
$$
\text{gcd}(x_i,y_j)=\text{gcd}(x_i,z_k)=\text{gcd}(y_j,z_k)=1, \quad i=1,\dots, n_1, \,\, j=1,\dots,n_2, \,\, k=1,\dots,n_3.
$$
We remark that we may have, for example, $\text{gcd}(x_i,x_j)>1$ for some $i\neq j$.

\begin{theorem}\label{th:3monred}
	There is an algorithm that, given an arbitrary $3$-monomial equation, reduces the problem of describing all integer solutions of this equation to the same problem for a finite number of three-monomial equations with independent monomials. Moreover, it is sufficient to find only primitive solutions to the resulting equations.
\end{theorem}

We will need the following lemma.
%
%
%\begin{proof}
%	$$
%	\sqrt[d]{r} = \sqrt[d]{\prod_{i=1}^k x_i^{e_i}} = \prod_{i=1}^k x_i^{e_i/d}.
%	$$
%	Because all $x_i$ and $e_i/d$ are integers, $\sqrt[d]{r}$ must be a rational number. Conversely, if $\sqrt[d]{r}$ is rational, let us represent its absolute value as 
%	$$
%	|\sqrt[d]{r}| = \prod_{p \in {\cal P}_r} p^{r_p},
%	$$
%	where ${\cal P}_r$ is a finite set of primes, and $r_p$ are non-zero integers. Because $e_1/d, \dots , e_k/d$ are non-zero integers, not all of the same sign, that do not share a common factor, equations
%	$$
%	\sum_{i=1}^k \frac{e_i}{d} z_{ip} = r_p
%	$$ 
%	are solvable in non-negative integers $z_{ip}$ for every $p \in {\cal P}_r$. Then for 
%	$$
%	x_i = \prod_{p \in {\cal P}_r} p^{z_{ip}}, \quad i=1,\dots, k,
%	$$
%	we have 
%	$$
%	\prod_{i=1}^k x_i^{e_i/d} = \prod_{i=1}^k \prod_{p \in {\cal P}_r} p^{z_{ip}e_i/d} = \prod_{p \in {\cal P}_r} \prod_{i=1}^k p^{z_{ip}e_i/d} = \prod_{p \in {\cal P}_r} p^{\sum_{i=1}^k (z_{ip}e_i/d)} = \prod_{p \in {\cal P}_r} p^{r_p} = |\sqrt[d]{r}|.
%	$$
%	Because integers $e_1/d, \dots , e_k/d$ do not share a common factor, at least one of these integers, say $e_j/d$, must be odd. Then, by changing the sign of $x_j$ if necessary, we may ensure that $\prod_{i=1}^k x_i^{e_i/d}=\sqrt[d]{r}$.
%\end{proof}

\begin{lemma}\label{le:ulemma}
	For any non-negative integers $e_1, \dots, e_N$ and non-zero integer $q$, there is a finite set ${\cal M}$  of $N$-tuples $d=(d_1, \dots, d_N)$ of positive integers such that
	\begin{itemize}
		\item[(i)] $\prod_{k=1}^{N} d_k^{e_k}$ is divisible by $q$ for every $d \in {\cal M}$, and
		\item[(ii)] for any non-zero integers $(U_1, \dots, U_N)$, $\prod_{k=1}^{N} U_k^{e_k}$ is divisible by $q$ if and only if there exists $d \in {\cal M}$ such that $U_k/d_k$ are integers for all $k=1,\dots, N$.
	\end{itemize}
	In particular, if $N=1$, then ${\cal M}$ can be chosen to be a one-element set. 
\end{lemma}
\begin{proof}
	Let ${\cal S}_q$ be the set of all $n$-tuples of non-zero integers $U=(U_1, \dots, U_{N})$ such that $\prod_{k=1}^{N} U_k^{e_k}$ is divisible by $q$. We say that $U=(U_1, \dots, U_{N})\in {\cal S}_q$ dominates $V=(V_1, \dots, V_{N})\in {\cal S}_q$ if for every $k=1,\dots,N$ ratio $U_k/V_k$ is an integer. We say that $U\in {\cal S}_q$ is minimal if $U$ dominates any $V\in {\cal S}_q$ only if $|V_k|=U_k$ for all $k$. Let ${\cal M}$ be the set of all minimal elements of ${\cal S}_q$. Because ${\cal M} \subset {\cal S}_q$, property (i) follows. Further, every $U \in {\cal S}_q$ dominates some $d=(d_1,\dots,d_N) \in {\cal M}$, hence $U_k/d_k$ are integers for all $k=1,\dots, N$. Conversely, if for some $d\in {\cal M}$ ratios $U_k/d_k$ are integers for all $k=1,\dots, N$, then  
	$$
	\prod_{k=1}^{N} U_k^{e_k} = \prod_{k=1}^{N} (U_k/d_k)^{e_k} \prod_{k=1}^{N} d_k^{e_k}.
	$$
	Because $\prod_{k=1}^{N} d_k^{e_k}$ is divisible by $q$, so is $\prod_{k=1}^{N} U_k^{e_k}$. Thus, $U \in {\cal S}_q$, and (ii) follows. 
	
	It is left to prove that ${\cal M}$ is a finite set. We claim that every $U\in {\cal M}$ satisfies the conditions
	\begin{equation}\label{eq:ucondab}
		\text{(a)} \quad U_k = 1 \quad \text{whenever} \quad e_k=0 \quad \text{and} \quad \text{(b)} \quad \prod_{k=1}^{N} U_k^{e_k} \leq q^{1+e_M},
	\end{equation}
	where $e_M=\max\{e_1, \dots, e_{N}\}$. Indeed, if $e_k=0$ then $U$ dominates 
	$$
	V=(U_1, \dots, U_{k-1}, 1, U_{k+1}, \dots, U_{N}) \in {\cal S}_q,
	$$ 
	hence, by the definition of minimality, we must have $U_k=|1|=1$, and (a) follows. Let us now check (b). For any prime $p$ let $E_p$ and $q_p$ be the exponents with which $p$ enters the prime factorizations of $\prod_{k=1}^{N} U_k^{e_k}$ and $q$, respectively. If (b) fails, then there must exist a prime $p$ such that $E_p > (1+e_M)q_p$. In particular, $E_p>0$, hence there exists $i$ such that $U_i$ is divisible by $p$, which implies that $E_p \geq e_i$. Further, because $U\in {\cal S}_q$, ratio $R=(1/q)\prod_{k=1}^{N} U_k^{e_k}$ is an integer. Prime $p$ enters the prime factorization of $R$ in the exponent $E_p - q_p$. If $q_p=0$, then $E_p-q_p = E_p \geq e_i$. If $q_p\geq 1$, then  
	$$
	E_p - q_p > (1+e_M)q_p - q_p = e_M q_p \geq e_M \geq e_i.
	$$
	Hence, in any case $E_p-q_p\geq e_i$, which implies that $R/p^{e_i}$ is an integer. But
	$$
	\frac{R}{p^{e_i}} = \frac{(1/q)\prod_{k=1}^{N_1} U_k^{e_k}}{p^{e_i}} = (1/q)(U_i/p)^{e_i}\prod_{k\neq i} U_k^{e_k}, 
	$$
	hence $(U_1, \dots, U_{i-1}, U_i/p, U_{i+1}, \dots, U_{N_1}) \in {\cal S}_q$, which is a contradiction with the minimality of $U$. Hence, \eqref{eq:ucondab} holds for any minimal $U \in S_q$. However, there are only finitely many $n$-tuples of positive integers satisfying \eqref{eq:ucondab}, hence the set ${\cal M}$ of minimal $U \in S_q$ is a finite set.
	
	Finally, let $N=1$, and let $q = \pm\prod_{p \in {\cal P}_q} p^{q_p}$ be the prime factorization of $q$. Then $U_1^{e_1}$ is divisible by $q$ if and only if $U_1^{e_1}$ is divisible by $p^{q_p}$ for every $p\in {\cal P}_q$. Equivalently, $U_1$ is divisible by $p^{q^*_p}$ for every $p\in {\cal P}_q$, where $q^*_p$ is the smallest integer satisfying $e_1 q^*_p \geq q_p$. Then the statement of the Proposition is true with ${\cal M}=\{q^*\}$, where $q^* = \prod_{p \in {\cal P}_q} p^{q^*_p}$.
\end{proof}

Now we are ready to prove Theorem \ref{th:3monred}.

\begin{proof} \emph{of Theorem \ref{th:3monred}.}
	A general three-monomial equation is an equation of the form
	\begin{equation}\label{eq:gen3mon}
		a \prod_{i=1}^n x_i^{\alpha_i} + b \prod_{i=1}^n x_i^{\beta_i} + c \prod_{i=1}^n x_i^{\gamma_i} = 0,
	\end{equation}
	where $\alpha_i,\, \beta_i,\,\gamma_i$ are non-negative integers, and $a,b,c$ are non-zero integers. 
	%We call a solution $(x_1,\dots,x_n)$ trivial if $\prod_{i=1}^n x_i=0$ and non-trivial otherwise. Trivial solutions are easy to find. 
	Let $(x_1,\dots,x_n)$ be any non-trivial solution to \eqref{eq:gen3mon}.  
	For any prime $p$, let $a_p$, $b_p$, $c_p$, $z_{1p}, \dots, z_{np}$ be the exponents with which $p$ enters the prime factorizations of $a,b,c,x_1,\dots,x_n$, respectively. Then $p$ enters the prime factorizations of the monomials of \eqref{eq:gen3mon} with the exponents
	$$
	a_p + \sum_{i=1}^n \alpha_i z_{ip}, \quad b_p + \sum_{i=1}^n \beta_i z_{ip}, \quad c_p + \sum_{i=1}^n \gamma_i z_{ip},
	$$
	respectively. By Lemma \ref{le:main}, we must have either
	\begin{equation}\label{eq:gen3moncase1}
		a_p + \sum_{i=1}^n \alpha_i z_{ip} = b_p + \sum_{i=1}^n \beta_i z_{ip} \leq c_p + \sum_{i=1}^n \gamma_i z_{ip},
	\end{equation}
	or 
	\begin{equation}\label{eq:gen3moncase2}
		a_p + \sum_{i=1}^n \alpha_i z_{ip} = c_p + \sum_{i=1}^n \gamma_i z_{ip} < b_p + \sum_{i=1}^n \beta_i z_{ip},
	\end{equation}
	or 
	\begin{equation}\label{eq:gen3moncase3}
		b_p + \sum_{i=1}^n \beta_i z_{ip} = c_p + \sum_{i=1}^n \gamma_i z_{ip} < a_p + \sum_{i=1}^n \alpha_i z_{ip}.
	\end{equation}
	Let ${\cal P}_1$, ${\cal P}_2$, ${\cal P}_3$ be the sets of primes satisfying \eqref{eq:gen3moncase1}, \eqref{eq:gen3moncase2} and \eqref{eq:gen3moncase3}, respectively. Note that ${\cal P}_1$, ${\cal P}_2$, ${\cal P}_3$ form a partition of the set ${\cal P}$ of all primes.
	
	Let ${\cal A}$ be the set of prime divisors of $abc$. For every $p\not\in {\cal A}$, $a_p=b_p=c_p=0$, and \eqref{eq:gen3moncase1}-\eqref{eq:gen3moncase3} simplify to
%	We start with the special case $a=b=c=1$, that is, with the equation of the form
%	\begin{equation}\label{eq:3mon1cof}
%		\prod_{i=1}^n x_i^{\alpha_i} + \prod_{i=1}^n x_i^{\beta_i} + \prod_{i=1}^n x_i^{\gamma_i} = 0,
%	\end{equation}
%	where $\alpha_i,\, \beta_i,\,\gamma_i$ are non-negative integers. For any prime $p$, let $z_{1p}, \dots, z_{np}$ be the powers with which $p$ enters the prime factorizations of $x_1,\dots,x_n$, respectively. Then $p$ enters the prime factorizations of the monomials of \eqref{eq:3mon1cof} in the powers
%	$$
%	A_p = \sum_{i=1}^n \alpha_i z_{ip}, \quad B_p = \sum_{i=1}^n \beta_i z_{ip}, \quad C_p = \sum_{i=1}^n \gamma_i z_{ip},
%	$$
%	respectively. Let $M=\min\{A_p,B_p,C_p\}$ be the smallest out of these three integers, and let $M'$ be the next smallest one. If $M < M'$, then, after cancelling $p^M$, we find that exactly two out of three monomials in \eqref{eq:3mon1cof} are divisible by $p$. But in this case their sum is not divisible by $p$ and therefore cannot be $0$, which is a contradiction.  This implies that $M=M'$, hence we have either
	\begin{equation}\label{eq:3mon1cofcase1}
		\sum_{i=1}^n \alpha_i z_{ip} = \sum_{i=1}^n \beta_i z_{ip} \leq \sum_{i=1}^n \gamma_i z_{ip},
	\end{equation}
	\begin{equation}\label{eq:3mon1cofcase2}
		\sum_{i=1}^n \alpha_i z_{ip} = \sum_{i=1}^n \gamma_i z_{ip} < \sum_{i=1}^n \beta_i z_{ip},
	\end{equation}
	\begin{equation}\label{eq:3mon1cofcase3}
		\sum_{i=1}^n \beta_i z_{ip} = \sum_{i=1}^n \gamma_i z_{ip} < \sum_{i=1}^n \alpha_i z_{ip}.
	\end{equation} 
	
The equations in \eqref{eq:3mon1cofcase1}-\eqref{eq:3mon1cofcase3} are homogeneous linear equations in non-negative integer variables. The study of such equations goes back to at least 1903 in the paper of Elliott \cite{elliott1903on}. 
A cornerstone concept of this theory is the concept of minimal solutions. For any  $x=(x_1,\dots,x_n)$ and $y=(y_1,\dots,y_n)$, we write $x<y$ (and say that $y$ dominates $x$) if $x_i\leq y_i$ for $i=1,\dots,n$ with at least one inequality being strict. Any homogeneous linear equation has a solution $(z_1,\dots,z_n)=(0,\dots,0)$ which we call trivial, and all other solutions in non-negative integers are called non-trivial. A non-trivial solution $z=(z_1,\dots,z_n)$ is called minimal if there is no other non-trivial solution $z'=(z'_1,\dots,z'_n)$ such that $z' < z$. It is well-known \cite{lambert1987borne}, that any homogeneous linear equation has a finite number of minimal solutions, and there are algorithms for listing them all. 
	
	Let 
	$
	(e_{k1}, \dots, e_{kn}),
	$ 
	$
	k=1,\dots, N_1,
	$ 
	$
	(f_{k1}, \dots, f_{kn}),
	$ 
	$
	k=1,\dots, N_2,
	$ 
	and 
	$
	(g_{k1}, \dots, g_{kn}),
	$ 
	$
	k=1,\dots, N_3, 
	$
	be the complete lists of minimal solutions to the equations in \eqref{eq:3mon1cofcase1}, \eqref{eq:3mon1cofcase2} and \eqref{eq:3mon1cofcase3}, respectively.
	Then all non-negative integer solutions to the equation in \eqref{eq:3mon1cofcase1} are given by 
	%$
	%z^j = \sum_{k=1}^N_1 u_{kj} e_k,
	%$
	$$
	z_{ip} = \sum_{k=1}^{N_1} u_{kp} e_{ki}, \quad i=1,2,\dots, n, \quad p \in {\cal P}_1,
	$$ 
	for non-negative integers $u_{1p}, \dots, u_{N_1p}$, see e.g. \cite{lambert1987borne}. The non-negative integer solutions to the equation in \eqref{eq:3mon1cofcase2} and \eqref{eq:3mon1cofcase3} can be written similarly.
	
	Let us now return to the equations in \eqref{eq:gen3moncase1}-\eqref{eq:gen3moncase3}. If $a_p\neq b_p$, then the equation in \eqref{eq:gen3moncase1} is an inhomogeneous linear equation in non-negative integers. As in the homogeneous case, a solution to such equation is called minimal if it does not dominate any other solution. The set of minimal solutions is finite and can be computed \cite{lambert1987borne}.
	Let $S_p^1$ be the set of minimal solutions to the equation in \eqref{eq:gen3moncase1} if $a_p\neq b_p$, and $S_p^1=\{0\}$ otherwise. Similarly, let $S_p^2$ and $S_p^3$ be the sets of minimal solutions to the equations in \eqref{eq:gen3moncase2} and \eqref{eq:gen3moncase3}, respectively, provided that they are inhomogeneous, and $S_p^2=\{0\}$ and $S_p^3=\{0\}$, respectively, in the homogeneous case. Then the general solutions to the equations in \eqref{eq:gen3moncase1}, \eqref{eq:gen3moncase2} and \eqref{eq:gen3moncase3}, are given by
	\begin{equation}\label{eq:gen3moncase1sol}
		z_{ip} = z_{ip}^0+\sum_{k=1}^{N_1} u_{kp} e_{ki}, \quad i=1,2,\dots, n, \quad z_p^0=(z_{1p}^0, \dots, z_{np}^0) \in S_p^1, \quad \quad p \in {\cal P}_1,
	\end{equation}
	\begin{equation}\label{eq:gen3moncase2sol}
		z_{ip} = z_{ip}^0+\sum_{k=1}^{N_2} u_{kp} f_{ki}, \quad i=1,2,\dots, n, \quad z_p^0=(z_{1p}^0, \dots, 	z_{np}^0) \in S_p^2, \quad \quad p \in {\cal P}_2,
	\end{equation}
	and
	\begin{equation}\label{eq:gen3moncase3sol}
		z_{ip} = z_{ip}^0+\sum_{k=1}^{N_3} u_{kp} g_{ki}, \quad i=1,2,\dots, n, \quad z_p^0=(z_{1p}^0, \dots, 	z_{np}^0) \in S_p^3, \quad \quad p \in {\cal P}_3,
	\end{equation}
	respectively, where $u_{kp}$ are arbitrary non-negative integers, see \cite{lambert1987borne}.	
%	be the set of minimal solutions to \eqref{eq:gen3moncase1} if $a_p\neq b_p$. If $a_p=b_p$, set $S_p^1=\{0\}$. Then the general solution to \eqref{eq:gen3moncase1} in non-negative integers is given by 
%	$$
%	z_{ip} = z_{ip}^0+\sum_{k=1}^{N_1} u_{kp} e_{ki}, \quad i=1,2,\dots, n, \quad z_p^0=(z_{1p}^0, \dots, z_{np}^0) \in S_p^1, \quad \quad p \in {\cal P}_1,
%	$$ 
%	where $u_{kp}$, $k=1,\dots,N_1$ are some non-negative integers, see \cite{lambert1987borne}. By similar argument, the general solution to \eqref{eq:gen3moncase2} is
%	$$
%	z_{ip} = z_{ip}^0+\sum_{k=1}^{N_2} u_{kp} f_{ki}, \quad i=1,2,\dots, n, \quad z_p^0=(z_{1p}^0, \dots, z_{np}^0) \in S_p^2, \quad \quad p \in {\cal P}_2,
%	$$ 
%	where $S_p^2$ is the set of minimal solutions to \eqref{eq:gen3moncase2} if $a_p\neq c_p$ and $S_p^2=\{0\}$ otherwise.
%	Finally, the general solution to \eqref{eq:gen3moncase3} is
%	$$
%	z_{ip} = z_{ip}^0+\sum_{k=1}^{N_3} u_{kp} g_{ki}, \quad i=1,2,\dots, n, \quad z_p^0=(z_{1p}^0, \dots, z_{np}^0) \in S_p^3, \quad \quad p \in {\cal P}_3,
%	$$ 
%	where $S_p^3$ is the set of minimal solutions to \eqref{eq:gen3moncase3} if $b_p\neq c_p$ and $S_p^3=\{0\}$ otherwise.

	Substituting prime factorizations
	$$
	|a| = \prod_{p \in {\cal P}} p^{a_p}, \quad |b| = \prod_{p \in {\cal P}} p^{b_p}, \quad |c| = \prod_{p \in {\cal P}} p^{c_p}, \quad |x_i| = \prod_{p \in {\cal P}} p^{z_{ip}}, \quad i=1,\dots, n
	$$
	into \eqref{eq:gen3mon} results in
	$$
	\pm \prod_{p \in {\cal P}} p^{a_p+\sum_{i=1}^n \alpha_i z_{ip}}
	\pm \prod_{p \in {\cal P}} p^{b_p+\sum_{i=1}^n \beta_i z_{ip}}
	\pm \prod_{p \in {\cal P}} p^{c_p+\sum_{i=1}^n \gamma_i z_{ip}} = 0,
	$$
 	for some combination of signs. For every $p \in {\cal P}_1$, \eqref{eq:gen3moncase1} implies that $p^{a_p+\sum_{i=1}^n \alpha_i z_{ip}} = p^{b_p+\sum_{i=1}^n \beta_i z_{ip}}$ is a common factor of all three monomials in this equation. Similarly, $p^{a_p+\sum_{i=1}^n \alpha_i z_{ip}}=p^{c_p+\sum_{i=1}^n \gamma_i z_{ip}}$ is a common factor for every $p \in {\cal P}_2$ by \eqref{eq:gen3moncase2}, while $p^{b_p+\sum_{i=1}^n \beta_i z_{ip}}=p^{c_p+\sum_{i=1}^n \gamma_i z_{ip}}$ is a common factor for every $p \in {\cal P}_3$ by \eqref{eq:gen3moncase3}. Cancelling all these common factors results in
	\begin{equation}\label{eq:cancelled}
 	\pm \prod_{p \in {\cal P}_3} p^{a_p-b_p+\sum_{i=1}^n (\alpha_i-\beta_i) z_{ip}}
 	\pm \prod_{p \in {\cal P}_2} p^{b_p-c_p+\sum_{i=1}^n (\beta_i-\gamma_i) z_{ip}}
 	\pm \prod_{p \in {\cal P}_1} p^{c_p-a_p+\sum_{i=1}^n (\gamma_i-\alpha_i) z_{ip}} = 0.
	\end{equation}
 	Inequalities in \eqref{eq:gen3moncase1}-\eqref{eq:gen3moncase3} imply that all three terms in the last expression are integers. Next, substituting expressions \eqref{eq:gen3moncase1sol}-\eqref{eq:gen3moncase3sol} for $z_{ip}$ results in
 	$$
 	A \prod_{p \in {\cal P}_3} p^{\sum_{i=1}^n (\alpha_i-\beta_i)\sum_{k=1}^{N_3} u_{kp} g_{ki}} +
 	B \prod_{p \in {\cal P}_2} p^{\sum_{i=1}^n (\beta_i-\gamma_i)\sum_{k=1}^{N_2} u_{kp} f_{ki}} +
 	C \prod_{p \in {\cal P}_1} p^{\sum_{i=1}^n (\gamma_i-\alpha_i)\sum_{k=1}^{N_1} u_{kp} e_{ki}} = 0,
 	$$
 	where
 	\begin{equation}\label{eq:ABCdef}
 		\begin{split}
 			& A = \pm \prod_{p \in {\cal P}_3} p^{a_p-b_p+\sum_{i=1}^n (\alpha_i-\beta_i) z_{ip}^0}, \\
 			& B = \pm \prod_{p \in {\cal P}_2} p^{b_p-c_p+\sum_{i=1}^n (\beta_i-\gamma_i) z_{ip}^0}, \\
 			& C = \pm \prod_{p \in {\cal P}_1} p^{c_p-a_p+\sum_{i=1}^n (\gamma_i-\alpha_i) z_{ip}^0}.
 		\end{split}
 	\end{equation}
	The last equation can be written as 
 	$$
 	A \prod_{p \in {\cal P}_3} p^{\sum_{k=1}^{N_3}u_{kp}g_k} +
 	B \prod_{p \in {\cal P}_2} p^{\sum_{k=1}^{N_2}u_{kp}f_k} +
	C \prod_{p \in {\cal P}_1} p^{\sum_{k=1}^{N_1}u_{kp}e_k} = 0,
 	$$
 	where 
	\begin{equation}\label{eq:abcdef3mon}
	g_k = \sum_{i=1}^n (\alpha_i-\beta_i) g_{ki}, \quad f_k = \sum_{i=1}^n (\beta_i-\gamma_i) f_{ki}, \quad e_k = \sum_{i=1}^n (\gamma_i-\alpha_i) e_{ki},
	\end{equation}
	or equivalently as
	\begin{equation}\label{eq:3mon1cofred}
		A \prod_{k=1}^{N_3} W_k^{g_k} 
		+ B \prod_{k=1}^{N_2} V_k^{f_k} 
		+ C \prod_{k=1}^{N_1} U_k^{e_k} = 0,
	\end{equation}
	where
	\begin{equation}\label{eq:UVWdef}
	U_k = \prod_{p \in {\cal P}_1} p^{u_{kp}}, \quad V_k=\prod_{p \in {\cal P}_2} p^{u_{kp}}, \quad W_k=\prod_{p \in {\cal P}_3} p^{u_{kp}}.
	\end{equation}

    For every original equation \eqref{eq:gen3mon}, there are only finitely many possible values for coefficients $(A,B,C)$ defined in \eqref{eq:ABCdef}, and all such possible triples can be explicitly listed. Indeed, for every prime $p$ that is not a divisor of $abc$, we have $a_p=b_p=c_p=0$ and $z_{ip}^0=0$ for $i=1,\dots,n$, hence all such primes contribute factor $1$ to the products in \eqref{eq:ABCdef}. Hence, the products are actually over prime factors of $abc$. There are finitely many such primes, and therefore finitely many ways how they can split into ${\cal P}_1, {\cal P}_2$ and ${\cal P}_3$. Also, for each $p$ there are only finitely many minimal solutions $(z_{1p}^0, \dots, z_{np}^0)$ to the equations in \eqref{eq:gen3moncase1}-\eqref{eq:gen3moncase3}. Finally, there are $8$ combinations of signs. Hence, there are finitely many possible triples of $(A,B,C)$, and we have reduced the original equation \eqref{eq:gen3mon} to finitely many equations of the form \eqref{eq:3mon1cofred} with integer variables $W_k$, $V_k$ and $U_k$.

	We remark that exponents $g_k,f_k,e_k$ defined in \eqref{eq:abcdef3mon} may be positive or negative integers, and coefficients $A,B,C$ are rational numbers. However, the three terms in \eqref{eq:3mon1cofred} are the same as the three terms in \eqref{eq:cancelled}, and therefore must be integers.
%	inequalities in \eqref{eq:3mon1cofcase1} - \eqref{eq:3mon1cofcase3} imply that all three summands in \eqref{eq:3mon1cofred} must be integers. Indeed, 
%	$$
%	\prod_{k=1}^{N_1} U_k^{c_k} = \prod_{k=1}^{N_1} \left(\prod_{p \in {\cal P}_1} p^{u_{kp}}\right)^{c_k} = \prod_{p \in {\cal P}_1} \prod_{k=1}^{N_1} p^{c_k u_{kp}} = \prod_{p \in {\cal P}_1} p^{\sum_{k=1}^{N_1} c_k u_{kp}}.
%	$$ 
%	To prove that the last expression is integer, it is sufficient to prove that the exponents are non-negative. Indeed,  
%	$$
%	\sum_{k=1}^{N_1} c_k u_{kp} = 
%	%\sum_{k=1}^{N_1} u_{kp}\left(\sum_{i=1}^n \gamma_i e_{ki} - \sum_{i=1}^n \alpha_i e_{ki}\right) =
%	\sum_{i=1}^n \sum_{k=1}^{N_1} \gamma_i u_{kp} e_{ki} - \sum_{i=1}^n \sum_{k=1}^{N_2} \alpha_i u_{kp} e_{ki} = \sum_{i=1}^n \gamma_i z_{ip} - \sum_{i=1}^n \alpha_i z_{ip} \geq 0,
%	$$
%	where the first equality follows from the definition of $c_k$ in \eqref{eq:abcdef3mon}, and the inequality follows from \eqref{eq:3mon1cofcase1}. The proof that $\prod_{k=1}^{N_3} W_k^{a_k}$ and $\prod_{k=1}^{N_2} V_k^{b_k}$ are integers is similar.
	
	Now, given integers $e_1, \dots, e_{N_1}$, which integers $m$ can be represented in the form $m = C \prod_{k=1}^{N_1} U_k^{e_k}$? 
	%for some integers $U_1, \dots, U_{N_1}$? 
	In other words, for which integers $m$ is equation
	\begin{equation}\label{eq:reduced1}
		\prod_{k=1}^{N_1} U_k^{e_k} = \frac{m}{C}
	\end{equation}
	solvable in integers $U_1, \dots, U_{N_1}$? 
	Let us consider cases (i) all $e_k\leq 0$, (ii) not all $e_k$ are of the same sign, and (iii) all $e_k \geq 0$. 
	
	\begin{itemize}
		\item In case (i), \eqref{eq:reduced1} is equivalent to $\prod_{k=1}^{N_1} U_k^{|e_k|} = \frac{C}{m}$ and may be solvable in integers only if $C$ is an integer and $m$ is a divisor of $C$. Hence, there are only finitely many possible values of $m$. 
		\item 	In case (ii), Proposition \ref{prop:2monex} implies that \eqref{eq:reduced1} is solvable if and only if $\sqrt[d]{m/C}$ is a rational number, where $d = \text{gcd}(e_1, \dots, e_{N_1})$. Representing $C=s/q$ and $\sqrt[d]{m/C}=\sqrt[d]{qm/s}=U/V$ as irreducible fractions, we obtain
		$$
		V^d q m = U^d s.
		$$  
		Because $\text{gcd}(U,V)=1$, $V^d$ must be a divisor of $s$, hence there are only finitely many possible values for $V$, say $v_1, \dots, v_K$. If $V=v_i$, then $m=\frac{U^d s}{v_i^d q}$. For $m$ to be an integer, we must have $U^d$ to be divisible by $q$. By the $N=1$ case of Lemma \ref{le:ulemma}, this happens if and only if $U$ is divisible by $q^*$, where $q^*$ is the smallest positive integer such that $(q^*)^d$ is divisible by $q$. 	
		Then substitution $U=q^* U'$ results in 
		\begin{equation}\label{eq:caseiirepl}
			m=\frac{s}{v_i^d}\frac{(q^*)^d}{q}(U')^d.
		\end{equation} 
		We remark that $U'$ is an integer variable and $\frac{s}{v_i^d}\frac{(q^*)^d}{q}$ is an integer coefficient. From \eqref{eq:reduced1}, we may find original variables $U_1, \dots, U_{N_1}$ as functions of $U'$. Indeed, \eqref{eq:reduced1} reduces to
		$$
			\prod_{k=1}^{N_1} U_k^{e_k/d} = e\sqrt[d]{\frac{m}{C}} = e \frac{U}{v_i} = e\frac{q^* U'}{v_i}, 
		$$ 
		where $e=\pm 1$ if $d$ is even and $e=1$ if $d$ is odd. This equation can be rewritten as a two-monomial equation \eqref{eq:gen2monred}, and its solution is given by \eqref{eq:gen2monredsol}. 
		\item In case (iii), if $C=s/q$ is an irreducible fraction, then $m=(s/q)\prod_{k=1}^{N_1} U_k^{e_k}$ is a non-zero integer if and only if $\prod_{k=1}^{N_1} U_k^{e_k}$ is a non-zero integer divisible by $q$. By Lemma \ref{le:ulemma}, there is a finite set ${\cal M}$  of $N_1$-tuples $d=(d_1, \dots, d_{N_1})$ of positive integers such that $(1/q)\prod_{k=1}^{N_1} U_k^{e_k}$ is an integer if and only if there exists $d \in {\cal M}$ such that $U'_k=U_k/d_k$ are integers for all $k=1,\dots, N_1$. Then
		\begin{equation}\label{eq:caseiiirepl}
			m = \frac{s\prod_{k=1}^{N_1} U_k^{e_k}}{q} = s \frac{\prod_{k=1}^{N_1} d_k^{e_k}}{q} \prod_{k=1}^{N_1} (U'_k)^{e_k} \quad \text{for some} \quad d \in {\cal M}.
		\end{equation} 
		Because $\prod_{k=1}^{N_1} d_k^{e_k}$ is divisible by $q$ by Lemma \ref{le:ulemma} (i), this expression involves integer variables $U'_k$ and integer coefficient.
	\end{itemize}
	
	%$\prod_{k=1}^{N_1} U_k^{e_k}$ must be the $d$-th power of an integer, where $d = \text{gcd}(e_1, \dots, e_{N_1})$. 
	Now, let us replace the term $C \prod_{k=1}^{N_1} U_k^{e_k}$ in \eqref{eq:3mon1cofred} by some divisor of $C$ in case (i), by \eqref{eq:caseiirepl} for some $i$ in case (ii), and by \eqref{eq:caseiiirepl} for some $d \in {\cal M}$ in case (iii). Let us do the same replacements for the other two terms in \eqref{eq:3mon1cofred}. This transforms \eqref{eq:3mon1cofred} into a finite number of three-monomial equations with integer variables and coefficients, non-negative exponents, and independent monomials.

%	In case (i), \eqref{eq:reduced1} is equivalent to $\prod_{k=1}^{N_1} U_k^{|e_k|} = \frac{C}{m}$ and can be solvable only if $C$ is an integer and $m$ is a divisor of $C$. Hence, there are only finitely many possible values of $m$. In case (ii), Lemma \ref{le:2mon} implies that \eqref{eq:reduced1} is solvable if and only if $\sqrt[d]{m/C}$ is a rational number, where $d = \text{gcd}(e_1, \dots, e_{N_1})$. Representing $C=s/q$ and $\sqrt[d]{m/C}=\sqrt[d]{qm/s}=U/V$ as irreducible fractions, we obtain
%	$$
%	V^d q m = U^d s.
%	$$  
%	Because $\text{gcd}(U,V)=1$, $V^d$ must be a divisor of $s$, hence there are only finitely many possible values for $V$, say $v_1, \dots, v_K$. If $V=v_i$, then $m=\frac{U^d s}{v_i^d q}$.
%	
%%$\prod_{k=1}^{N_1} U_k^{e_k}$ must be the $d$-th power of an integer, where $d = \text{gcd}(e_1, \dots, e_{N_1})$. 
%	Now, let us replace the term $C \prod_{k=1}^{N_1} U_k^{e_k}$ in \eqref{eq:3mon1cofred} by a divisor of $C$ in case (i) and by $\frac{U^d s}{v_i^d q}$ in case (ii). In case (iii), we keep this term unchanged. Let us do the same replacements for the other two terms in \eqref{eq:3mon1cofred}. 
%	This transforms \eqref{eq:3mon1cofred} into a finite number of three-monomial equations with independent monomials. 

	We remark that all prime factors of $U_k$, $V_k$ and $W_k$ belong to the disjoint classes ${\cal P}_1$, ${\cal P}_2$, and ${\cal P}_3$, respectively, hence it is sufficient to find only primitive solutions to these equations. Once all these equations are solved, it is easy to find integer solutions to \eqref{eq:3mon1cofred}, and then absolute values of the solutions to the original equation \eqref{eq:gen3mon} are
	$$
	|x_i| = \prod_{p \in {\cal P}} p^{z_{ip}} = \prod_{p \in {\cal P}_1} p^{z_{ip}^0+\sum_{k=1}^{N_1}u_{kp}e_{ki}} \prod_{p \in {\cal P}_2} p^{z_{ip}^0+\sum_{k=1}^{N_1}u_{kp}f_{ki}} \prod_{p \in {\cal P}_3} p^{z_{ip}^0+\sum_{k=1}^{N_1}u_{kp}g_{ki}} = 
	$$
	$$
	= \left(\prod_{p \in {\cal P}_1} p^{z_{ip}^0}\prod_{p \in {\cal P}_2} p^{z_{ip}^0}\prod_{p \in {\cal P}_3} p^{z_{ip}^0}\right) \prod_{k=1}^{N_1} U_k^{e_{ki}} \prod_{k=1}^{N_2} V_k^{f_{ki}} \prod_{k=1}^{N_3} W_k^{g_{ki}}, \quad i=1,\dots, n,
	$$
	where we have used \eqref{eq:gen3moncase1sol}-\eqref{eq:gen3moncase3sol} and \eqref{eq:UVWdef}. Once $|x_i|$ are found, it is easy to find the signs of the variables and finish the solution of \eqref{eq:gen3mon}.	
%	The same method can be applied to the general three-monomial equation \eqref{eq:gen3mon}. 
%
%	
%	Now let $p$ be a divisor of $abc$. Assume, for example, that \eqref{eq:gen3moncase1} holds. Then $(z_{1p}, \dots, z_{np})$ dominates some minimal solution $(z_{1p}^0, \dots, z_{np}^0)$ to the equation in \eqref{eq:gen3moncase1}. 
%	%$$
%	%\sum_{i=1}^n \alpha_i (z_{ip}-z_{ip}^0) = \sum_{i=1}^n \beta_i (z_{ip}-z_{ip}^0)
%	%$$
%	%for some minimal solution $(z_{1p}^0, \dots, z_{np}^0)$ to the equation in \eqref{eq:gen3moncase1}.  
%	Then do the substitution $x_i = p^{z_{ip}^0} x'_i$ for some new integer variables $x'_i$.  Let $z'_{1p}, \dots, z'_{np}$ be the powers with which $p$ enters the prime factorizations of $x'_1,\dots,x'_n$, respectively. Then $z'_{ip}=z_{ip}-z_{ip}^0$, $i=1,\dots, n$, and by subtracting equality $a_p + \sum_{i=1}^n \alpha_i z_{ip}^0 = b_p + \sum_{i=1}^n \beta_i z_{ip}^0$ from the equation in \eqref{eq:gen3moncase1}, we obtain
%	$$
%	\sum_{i=1}^n \alpha_i z'_{ip} = \sum_{i=1}^n \beta_i z'_{ip}.
%	$$
%	
%	Let ${\cal A}$ be the (finite) set of all prime divisors of $abc$. For each $p \in {\cal A}$, we have a finite number (say $K_p$) of minimal solutions to equations in \eqref{eq:gen3moncase1}-\eqref{eq:gen3moncase3}. In total, we have $\prod_{p \in {\cal A}} K_p$ cases, and in each case, let us do the described substitutions for all $p \in {\cal A}$. The equations after substitutions can be reduced to equations with independent monomials exactly as we did in the $a=b=c=1$ case.
\end{proof}

In many examples, the resulting equations with independent monomials are trivial, hence the reduction in the proof of Theorem \ref{th:3monred} instantly solves the original equation. As an illustration, assume that $\alpha_i,\, \beta_i,\,\gamma_i$ in \eqref{eq:gen3mon} are such that both systems \eqref{eq:gen3monsyst1} and \eqref{eq:gen3monsyst2} are solvable in non-negative integers, so that \eqref{eq:gen3mon} can be instantly solved by Proposition \ref{prop:3monformula}. Then the first equation in \eqref{eq:gen3monsyst1} implies that 
$$
z_i = \sum_{k=1}^{N_1} u_k e_{ki}, \quad i=1,\dots, n, 
$$
where $u_1, \dots, u_{N_1}$ are some non-negative integers, and 
$
(e_{k1}, \dots, e_{kn}),
$ 
$
k=1,\dots, N_1,
$ 
are the minimal solutions to this equation. 
%, see the proof of Proposition \ref{prop:3monformula}. 
Then, by \eqref{eq:gen3monsyst1},
\begin{equation}\label{eq:3monsum1}
	1 = \sum_{i=1}^n(\gamma_i-\alpha_i) z_i = \sum_{i=1}^n(\gamma_i-\alpha_i) \sum_{k=1}^{N_1} u_k e_{ki} = \sum_{k=1}^{N_1} u_k \sum_{i=1}^n (\gamma_i-\alpha_i) e_{ki} = \sum_{k=1}^{N_1} u_k e_k, 
\end{equation}
where the last equality follows from \eqref{eq:abcdef3mon}. By similar argument, the solvability of \eqref{eq:gen3monsyst2} implies that
\begin{equation}\label{eq:3monsumm1}
	-1 = \sum_{k=1}^{N_1} u'_k e_k
\end{equation} 
for some non-negative integers $u'_k$. This implies that not all $e_k$ are of the same sign, and $d=\text{gcd}(e_1,\dots,e_{N_1})=1$, hence the algorithm in the proof of Theorem \ref{th:3monred} replaces the term $C \prod_{k=1}^{N_1} U_k^{e_k}$ in \eqref{eq:3mon1cofred} by $C' (U')^{d} = C' U'$ for some integer coefficient $C'$ and integer variable $U'$. In conclusion, the original equation is reduced to a finite number of trivial equations of the form \eqref{eq:sepvara} with $U'$ in place of $x_i$. In other words, Proposition \ref{prop:3monformula} corresponds to the trivial case of Theorem \ref{th:3monred} when the reduction is to equations of the form \eqref{eq:sepvara}.

Now assume that system \eqref{eq:gen3monsyst1} is solvable in non-negative integers, while system  \eqref{eq:gen3monsyst2} is not. In this case, \eqref{eq:3monsum1} is true but \eqref{eq:3monsumm1} is not. By \eqref{eq:3monsum1}, $\text{gcd}(e_1,\dots,e_{N_1})=1$, hence \eqref{eq:3monsumm1} may fail only if all $e_k$ are non-negative. But then \eqref{eq:3monsum1} implies that $e_j=1$ for some $1\leq j \leq n$. If $e_i=0$ for all $i\neq j$, then term $C \prod_{k=1}^{N_1} U_k^{e_k}$ in \eqref{eq:3mon1cofred} reduces to $C U_j$, hence the reduced equation is of the form \eqref{eq:sepvara}. If, conversely, $e_i>0$ for some $i\neq j$, then term $C \prod_{k=1}^{N_1} U_k^{e_k}$ in \eqref{eq:3mon1cofred} contains multiplier $U_j U_i^{e_i}$, hence the whole equation is of the form \eqref{eq:x1kx2}, and can be solved as explained in Section \ref{sec:solrepr}. We have just proved the following result.

\begin{proposition}\label{prop:3monsuffcond}
	There is an algorithm for describing all integer solutions to any three-monomial Diophantine equation that can be written in the form \eqref{eq:gen3mon2} such that system \eqref{eq:gen3monsyst1} is solvable in non-negative integers. 
\end{proposition}

\subsection{Examples}

\subsubsection{Some simple examples}

In this section, we give simple examples of three-monomial equations for which 
\begin{itemize}
	\item[(a)] Proposition \ref{prop:3monformula} is applicable;
	\item[(b)] Proposition \ref{prop:3monformula} is not applicable, but Proposition \ref{prop:3monsuffcond} is;
	\item[(c)] Proposition \ref{prop:3monsuffcond} is not applicable, but the algorithm in the proof of Theorem \ref{th:3monred} can still solve the equation completely.
\end{itemize}

%Let us now consider some examples. 
Obviously, for each three-monomial equation, there are three ways to write it in the form \eqref{eq:gen3mon2}. Proposition \ref{prop:3monsuffcond} states that if system \eqref{eq:gen3monsyst1} is solvable for at least one of these three ways, then the algorithm works. 
As a simple example, equation 
\begin{equation}\label{eq:x3my2zmzinitial}
	x^3-y^2z-z=0
\end{equation}
can be written in the form \eqref{eq:gen3mon2} as
$$
x^3-y^2z=z.
$$
If we rename variables as $x=x_1$, $y=x_2$ and $z=x_3$, this is \eqref{eq:gen3mon2} with 
$$
(\alpha_1,\alpha_2,\alpha_3,\beta_1,\beta_2,\beta_3,\gamma_1,\gamma_2,\gamma_3)=(3,0,0,0,2,1,0,0,1).
$$ 
Then system \eqref{eq:gen3monsyst1} 
$$
3z_1 = 2z_2+z_3 = z_3 - 1 
$$ 
has no solutions in non-negative integers. However, the same equation \eqref{eq:x3my2zmzinitial} can also be written in the form \eqref{eq:gen3mon2} as 
$$
y^2z+z=x^3.
$$
In this case, $(\alpha_1,\alpha_2,\alpha_3,\beta_1,\beta_2,\beta_3,\gamma_1,\gamma_2,\gamma_3)=(0,2,1,0,0,1,3,0,0)$. Then systems \eqref{eq:gen3monsyst1} and \eqref{eq:gen3monsyst2} are
$$
2z_2+z_3 = z_3 = 3z_1 - 1 \quad \text{and} \quad 2t_2+t_3 = t_3 = 3t_1 + 1,
$$ 
respectively. These systems are solvable in non-negative integers, hence Proposition \ref{prop:3monformula} is applicable. For example, we may take $(z_1,z_2,z_3)=(1,0,2)$ and $(t_1,t_2,t_3)=(0,0,1)$. Substituting this into the general formula \eqref{eq:gen3mon2sol}, we obtain 
$$
(x,y,z) = \left(\frac{u_2^2u_3+u_3}{w}u_1, \, u_2, \, \frac{(u_2^2u_3+u_3)^2u_1^3}{w^3}u_3\right), \,\, u_i \in {\mathbb Z}, \,\, w \in D_1(u_2^2u_3+u_3) \cap D_1(u_1^3).
$$

As a next example, let us consider the equation
\begin{equation}\label{eq:x3my2zmy}
	x^3-y^2z-y=0,
\end{equation}
which can be written in the form \eqref{eq:gen3mon2} as $y^2z+y=x^3$, or $x^3-y^2z=y$, or $x^3-y=y^2z$. 
%If we rename variables as $x=x_1$, $y=x_2$ and $z=x_3$, 
Up to the names of the variables, this is \eqref{eq:gen3mon2} with $(\alpha_1,\alpha_2,\alpha_3,\beta_1,\beta_2,\beta_3,\gamma_1,\gamma_2,\gamma_3)$ equal to $(0,2,1,0,1,0,3,0,0)$,  $(3,0,0,0,2,1,0,1,0)$ and $(3,0,0,0,1,0,0,2,1)$, respectively. The corresponding systems \eqref{eq:gen3monsyst1} and \eqref{eq:gen3monsyst2} are
$$
2z_2+z_3 = z_2 = 3 z_1 - 1 \quad \text{and} \quad 2t_2+t_3 = t_2 = 3 t_1 + 1,
$$ 
$$
3 z_1 = 2z_2+z_3 = z_2 - 1 \quad \text{and} \quad 3 t_1 = 2t_2+t_3 = t_2 + 1,
$$ 
and
$$
3 z_1 = z_2 = 2z_2+z_3 - 1 \quad \text{and} \quad 3 t_1 = t_2 = 2t_2+t_3 + 1,
$$
respectively. It is easy to see that only the last system in $z_i$ is solvable in non-negative integers, with, for example $(z_1,z_2,z_3)=(0,0,1)$ being a solution. However, the corresponding system $3 t_1 = t_2 = 2t_2+t_3 + 1$ has no solutions in non-negative integers. Hence, Proposition \ref{prop:3monformula} is not applicable to equation \eqref{eq:x3my2zmy}. However, Proposition \ref{prop:3monsuffcond} is still applicable, and it guarantees that \eqref{eq:x3my2zmy} can be solved by reduction to equation(s) of the form \eqref{eq:x1kx2}. And indeed, we can rewrite \eqref{eq:x3my2zmy} as $y(yz+1)=x^3$. Integers $y$ and $yz+1$ are coprime, and their product can be a perfect cube only if they both are perfect cubes. Then $y=v^3$ and $yz+1=u^3$ for some integers $u$ and $v$. This implies that
\begin{equation}\label{eq:x3my2zmyred}
	v^3 z + 1 = u^3.
\end{equation} 
This equation is of the form \eqref{eq:x1kx2} and it is easy to solve: just let $u \in {\mathbb Z}$ be arbitrary, let $v$ be any element of set $D_3(u^3-1)$, and then express $z$ as $z=\frac{u^3-1}{v^3}$. Then $x=\sqrt[3]{(yz+1)y}=\sqrt[3]{u^3v^3}=uv$, so the final answer is
$$
(x,y,z) = \left(uv, \, v^3, \, \frac{u^3-1}{v^3}\right), \quad u \in {\mathbb Z}, \quad v \in D_3(u^3-1).
$$    

Our next example is the equation
\begin{equation}\label{eq:xpx2ymyz2}
	x+x^2y-yz^2=0.
\end{equation}
It can be written in the form \eqref{eq:gen3mon2} as $x+x^2y=yz^2$, or $x-yz^2=-x^2y$, or $x^2y-yz^2=-x$. Up to the names of the variables, this is \eqref{eq:gen3mon2} with $(\alpha_1,\alpha_2,\alpha_3,\beta_1,\beta_2,\beta_3,\gamma_1,\gamma_2,\gamma_3)$ equal to  $(1,0,0,2,1,0,0,1,2)$, $(1,0,0,0,1,2,2,1,0)$,  and $(2,1,0,0,1,2,1,0,0)$, respectively. The corresponding systems \eqref{eq:gen3monsyst1} are 
$$
z_1 = 2z_1+z_2 = z_2+2z_3 - 1, 
$$
$$
z_1 = z_2+2z_3 = 2z_1+z_2 - 1,
$$
and
$$
2z_1+z_2 = z_2+2 z_3 = z_1 -1,
$$ 
respectively. It is easy to see that none of these systems have a solution in non-negative integers, hence Proposition \ref{prop:3monsuffcond} is not applicable to this equation. However, we can still use the reduction in the proof of Theorem \ref{th:3monred} to solve the equation. Let $(x,y,z)$ be any solution to \eqref{eq:xpx2ymyz2} with $xyz \neq 0$. For any prime $p$, let $x_p, y_p$ and $z_p$ be the powers with which $p$ enters the prime factorizations of $x,y,z$, respectively. Then $p$ enters the prime factorizations of the monomials of \eqref{eq:xpx2ymyz2} in the powers $x_p$, $2x_p+y_p$ and $y_p+2z_p$, respectively.
By Lemma \ref{le:main}, one of the following systems is true
$$
	x_p = 2x_p + y_p \leq y_p + 2z_p, \quad x_p = y_p + 2z_p < 2x_p + y_p, \quad 2x_p + y_p = y_p + 2z_p < x_p.
$$
Let ${\cal P}_1$, ${\cal P}_2$, ${\cal P}_3$ be the sets of primes satisfying each of these systems. Note that ${\cal P}_3$ is an empty set, because inequality $2x_p + y_p < x_p$ is impossible for non-negative $x_p, y_p$. 
Now, let 
$$
u_x = \prod_{p \in {\cal P}_1} p^{x_p}, \, v_x = \prod_{p \in {\cal P}_2} p^{x_p}, \, 
u_y = \prod_{p \in {\cal P}_1} p^{y_p}, \, v_y = \prod_{p \in {\cal P}_2} p^{y_p}, \, 
u_z = \prod_{p \in {\cal P}_1} p^{z_p}, \, v_z = \prod_{p \in {\cal P}_2} p^{z_p}.
$$
Then
$$
x = e_x u_x v_x, \quad y = e_y u_y v_y, \quad z = e_z u_z v_z,
$$
where each $e_x$, $e_y$, $e_z$ is either $1$ or $-1$. 
For primes in ${\cal P}_1$, equation $x_p = 2x_p + y_p$ implies that $x_p=y_p=0$, hence $u_x=u_y=1$, thus $x = e_x v_x$, $y = e_y v_y$ and $z = e_z u_z v_z$. Substituting this into \eqref{eq:xpx2ymyz2}, we obtain 
\begin{equation}\label{eq:xpx2ymyz2red}
	e_x v_x + e_y v_x^2 v_y - e_y v_y u_z^2 v_z^2 = 0.
\end{equation}
For primes in ${\cal P}_2$, we have $x_p = y_p+2z_p$, thus 
$$
v_x = \prod_{p \in {\cal P}_2} p^{x_p} = \prod_{p \in {\cal P}_2} p^{y_p+2z_p} = \prod_{p \in {\cal P}_2} p^{y_p} \left(\prod_{p \in {\cal P}_2} p^{z_p}\right)^2 = v_y v_z^2.
$$
After dividing both parts of \eqref{eq:xpx2ymyz2red} by $v_x = v_y v_z^2$, we obtain
$$
e_x + e_y v_y^2 v_z^2 - e_y u_z^2 = 0,
$$
which is an equation with independent monomials in variables $v_y$, $v_z$, and $u_z$. After multiplying it by $e_x$, we can rewrite it as 
\begin{equation}\label{eq:xpx2ymyz2redf}
	1 + e v_y^2 v_z^2 - e u_z^2 = 0,
\end{equation}
where $e=e_x e_y$.

We followed the steps in the proof of Theorem \ref{th:3monred} with the only aim to illustrate the general method. For individual equations, the reduction can often be performed much easier by equation-specific arguments. For example, for equation \eqref{eq:xpx2ymyz2}, we can note that $x=y(z^2-x^2)$, hence $x=yt$ for integer $t=z^2-x^2=z^2-(yt)^2$, or equivalently $t(1+y^2t)=z^2$. Because factors $t$ and $1+y^2t$ are coprime, their product can be a perfect square only if $t = eu^2$ and $1+y^2t = ev^2$, where $u,v$ are some integers and $e$ is either $1$ or $-1$. Then $1+y^2(eu^2) = ev^2$, or
$$
1 + ey^2 u^2 - e v^2 = 0.
$$
Up to the names of the variables, this is exactly the equation \eqref{eq:xpx2ymyz2redf}. In this example, it can be easily solved by noticing that $(yu)^2$ and $v^2$ must be consecutive perfect squares, but the only consecutive perfect squares are $0$ and $1$, hence either $yu=0$ or $v=0$. 
%Hence, \eqref{eq:xpx2ymyz2} has no integer solutions with $xyz \neq 0$. 
If $yu=0$, then $x=yt=y(eu^2)=0$. If $v=0$, then $e=-1$, $|y|=|u|=1$, $t=eu^2=-1$, and $x=yt=-y=-(\pm 1)$. In conclusion, all integer solutions to \eqref{eq:xpx2ymyz2} are given by  
$$
	(x,y,z) = (0,0,w), (0,w,0), (-1,1,0) \quad \text{and} \quad (1,-1,0), \quad w \in {\mathbb Z}.
$$

\subsubsection{Generalized Fermat equations}

%We start with examples for which Proposition \ref{prop:3monformula} is applicable. 
Our next example is the equation 
\begin{equation}\label{eq:x2py3mz5}
	x^2 + y^3 = z^5.
\end{equation}
This equation is known as ``the icosahedral case'' of the equation $x^n+y^m=z^k$. For such equations, it is an active research area to determine \emph{primitive} solutions, that is, ones with $\text{gcd}(x,y,z)=1$. Non-primitive solutions are usually excluded from the consideration on the basis that it is trivial to construct infinite families of such solutions. In 2005, Edwards \cite{edwards2005platonic} described all primitive solutions to \eqref{eq:x2py3mz5}. Up to changing $x$ into $-x$, there are exactly $27$ distinct parametrizations of such solutions.

We will describe \emph{all} integer solutions to \eqref{eq:x2py3mz5} using Proposition \ref{prop:3monformula}. The solutions with $xyz=0$ are easy to describe, so we may assume that $xyz\neq 0$. Equation \eqref{eq:x2py3mz5} is of the form \eqref{eq:gen3mon2} with $x_1=x$, $x_2=y$, $x_3=z$,  $a=b=c=1$, $\alpha_1=2$, $\beta_2=3$, $\gamma_3=5$, and $\alpha_2=\alpha_3=\beta_1=\beta_3=\gamma_1=\gamma_2=0$. Systems \eqref{eq:gen3monsyst1} and \eqref{eq:gen3monsyst2} reduce to $2z_1=3z_2=5z_3-1$ and $2t_1=3t_2=5t_3+1$. Both systems are solvable in non-negative integers, with the smallest solutions being $(z_1,z_2,z_3)=(12,8,5)$ and $(t_1,t_2,t_3)=(3,2,1)$, respectively. Hence, \eqref{eq:gen3mon2sol} reduces to
$$
x = \frac{\left(u_1^2+u_2^3\right)^{12}u_3^{15}}{w^{15}} u_1, \quad y = \frac{\left(u_1^2+u_2^3\right)^{8}u_3^{10}}{w^{10}} u_2, \quad z = \frac{\left(u_1^2+u_2^3\right)^{5}u_3^5}{w^6} u_3,
$$
where
$$
u_1, u_2, u_3 \in {\mathbb Z}, \quad 
w \in D_1\left(u_1^2+u_2^3\right) \cap D_1\left(u_3^5\right).
$$

In general, Proposition \ref{prop:3monformula} is applicable to the equation 
\begin{equation}\label{eq:genFermatcof1}
	x^n + y^m = z^k
\end{equation}
if and only if exponents $n,m,k$ satisfy
$$
\min\left({\text{gcd}(nm,k),\text{gcd}(nk,m),\text{gcd}(mk,n)}\right)=1,
$$
that is, there exists an exponent coprime with the other two. If $d=\text{gcd}(n,m,k)\geq 3$, then \eqref{eq:genFermatcof1} has no solutions with $xyz\neq 0$ by FLT. The case $d=2$ requires more care, but the most problematic case is when $n=ab$, $m=ac$, $k=bc$ for some distinct pairwise coprime positive integers $a,b,c$. In this case, Theorem \ref{th:3monred} reduces \eqref{eq:genFermatcof1} to the problem of finding primitive solutions to the same equation. The latter solutions are known in the case $(a,b,c)=(1,2,t)$ and its permutations, see \cite{merel1997winding}. All other cases would follow from Tijdeman-Zagier conjecture, which predicts that \eqref{eq:genFermatcof1} has no primitive solutions with $xyz\neq 0$ provided that $\min(n,m,k)\geq 3$. This conjecture is very difficult in general, and a million-dollars prize is offered for its resolution \cite{beal1997generalization}. In light of the above discussion, it is interesting to investigate whether Tijdeman-Zagier conjecture becomes any easier in the special case $(n,m,k)=(ab,ac,bc)$ for pairwise coprime $a,b,c$. 
%In this case, \eqref{eq:genFermatcof1} is equivalent to $(x^a/z^c)^b + (y^a/z^b)^c=1$, hence the question reduces to investigating rational points on the curve $X^b + Y^c = 1$.
In this case, \eqref{eq:genFermatcof1} is equivalent to $1+(y^c/x^b)^a=(z^c/x^a)^b$, hence the question reduces to investigating rational solutions $(X,Y)$ to the equation 
$$
X^b - Y^a = 1.
$$
We remark that the famous Catalan's conjecture proved by Mih\u{a}ilescu \cite{MR2076124} in 2004 states that the last equation has no solutions in integers $X,Y,a,b \geq 2$ other than $3^2-2^3=1$.
%We remark that the last equation has no solutions in integers $X,Y,a,b \geq 2$ other than $3^2-2^3=1$ - this is the famous Catalan's conjecture proved by Mih\u{a}ilescu \cite{MR2076124} in 2004. 
However, allowing $X,Y$ to be rational makes the problem more difficult. See \cite[Theorem 12.4]{MR891406} and \cite[Theorem 5]{mihailescu2007cylcotomic} for some partial results in this direction.

\subsubsection{Cyclic three-monomial equations}

Our next example is the resolution of all three-monomial equations obeying certain symmetry. Let us call equation $P(x_1, x_2, \dots, x_n)=0$ \emph{cyclic} if 
$$
P(x_1, x_2, \dots, x_n) = P(x_2, \dots, x_n, x_1).
$$
It is easy to see that the only irreducible cyclic equations with three monomials are (i) three-monomial equations in one variable, (ii) equations of the form $ax^n + bx^my^m + ay^n=0$ for some integers $a,b$ and non-negative integers $n,m$, and (iii) equations of the form 
\begin{equation}\label{eq:shiftsymm}
x^a y^b + y^a z^b + z^a x^b = 0,
\end{equation}  
for some non-negative integers $a,b$.
Equations in case (i) are trivial, ones in case (ii) are covered in Section \ref{sec:3mon2var}, so it is left to solve \eqref{eq:shiftsymm}. If $d=\text{gcd}(a,b)=2$, then all monomials in \eqref{eq:shiftsymm} are non-negative, hence the only solution to \eqref{eq:shiftsymm} is $x=y=z=0$. If $d\geq 3$, then \eqref{eq:shiftsymm} has no solutions with $xyz\neq 0$ by FLT \cite{wiles1995modular}. Hence, we may assume that $a,b$ are coprime. 

For any prime $p$, let $x_p$, $y_p$ and $z_p$ be the exponents with which $p$ enters the prime factorizations of $x$, $y$ and $z$, respectively. Then $p$ enters the prime factorizations of the monomials in \eqref{eq:shiftsymm} with the exponents $ax_p+by_p$, $ay_p+bz_p$ and $az_p+bx_p$. Hence, by Lemma \ref{le:main}, we have either (i) $ax_p+by_p=ay_p+bz_p \leq az_p+bx_p$, or (ii) $ax_p+by_p=az_p+bx_p < ay_p+bz_p$, or (iii) $ay_p+bz_p=az_p+bx_p < ax_p+by_p$. In case (i), we have $a(x_p-y_p)=b(z_p-y_p)$. Because $a$ and $b$ are coprime, this implies that $x_p-y_p=k_p b$ and $z_p-y_p=k_p a$ for some integer $k_p$. Then $z_p-x_p=k_p(a-b)$.  
After dividing all monomials in \eqref{eq:shiftsymm} by $p^{ax_p+by_p}=p^{ay_p+bz_p}$, we obtain that $p$ enters the prime factorization of the third monomial only, and the corresponding exponent is 
$$
(az_p+bx_p)-(ax_p+by_p) = a(z_p-x_p)+b(x_p-y_p)=
%k_p(a(a-b)+b^2)=
k_p(a^2-ab+b^2).
$$
The cases (ii) and (iii) are similar. In conclusion, after cancelling all common factors of the monomials of \eqref{eq:shiftsymm}, we obtain three monomials, and each of them is the $m$-th power of an integer, where $m=a^2-ab+b^2$. If $m \geq 3$, then $xyz=0$ by FLT. If $m=a^2-ab+b^2 < 3$, then $(a,b)=(1,1)$ or $(1,0)$ or $(0,1)$. In these cases equation \eqref{eq:shiftsymm} is easy to solve.

\subsubsection{Three-monomial equations of small degree}

It is natural to classify the equations by degree. Let us say that two equations belong to the same family if they differ only by coefficients. It is obvious that, up to the names of the variables, there is only a finite number of families of equations with given degree and number of monomials. The only families of quadratic three-monomial equations not solvable by Proposition \ref{prop:3monsuffcond} are $ax^2+by^2+c=0$ and  $ax^2+by^2+cz^2=0$. The first family is essentially the general Pell's equation, whose solution methods are well-known, see e.g. \cite[Section 10.5]{MR4368213}. The second family has been studied by Legendre, see \cite[Section 10.7]{MR4368213} for a modern treatment. In conclusion, all quadratic three-monomial equations are easy, and the first interesting case is cubic equations.

\begin{table}
\footnotesize
\begin{center}
\begin{tabular}{ | c | c || c|c|}
	\hline
	\textbf{Equation \eqref{eq:gen3mon2}} & \textbf{Solution to \eqref{eq:gen3monsyst1} } &  \textbf{Equation \eqref{eq:gen3mon2}} & \textbf{Solution to \eqref{eq:gen3monsyst1} }  \\
	\hline \hline
	$ax^2y+by=cz^2$&$(0,1,1)$& $ax+bz^2=cx^2y$&$(0,1,0)$\\\hline
	$ax^3+by=cyz$&$(0,0,1)$& $a t x + b t y = c x y z$&$(0,0,1,0)$\\\hline
	$a x^2 + b x y z = c t y$&$(0,0,0,1)$& $ a x^2 + by^2 = c x y z$&$(0,0,1)$\\\hline
	$a x^2 y + b y z = c t z$&$(0,0,0,1)$& $a x^2 y + b x z = c t z$&$(0,0,0,1)$\\\hline
	$a t y + b x^2 y = c x z$&$(0,0,1,0)$& $a x^2 y + b x z = c y z$&$(0,1,1)$\\\hline
	$a t^2 + b x^2 y = c y z$&$(0,0,1,0)$& $a t^2 + b x^2 y = c x z$&$(0,0,1,0)$\\\hline
	$a x^2 y + b y z = c z^2$&$(1,1,2)$& $a x^2 y + b x z = c z^2$&$(0,1,1)$\\\hline
	$a x^2 y + b x y = c z^2$&$(0,1,1)$& $a x^2 y + b x z = c y^2$&$(0,1,1)$\\\hline
	$a y^2+ b z^2 = cx^2 y$&$(1,1,1)$& $a x^2 y + b x^2 = c y z$&$(0,0,1)$\\\hline
	$a x^2+ bz^2 = cx^2 y$&$(0,1,0)$& $ax^3 + b t y = c y z$&$(0,0,1,0)$\\\hline
	$a x^3 + b x y = c y z$&$(0,0,1)$& $a x^3 + b x y = c z^2$&$(1,2,2)$\\\hline
	$a x^3 + b y^2 = c y z$&$(0,0,1)$& $a y^2+ b z^2 = c x^3$&$(1,1,1)$\\\hline
	$a y^2 z+ b z = c x^3$&$(1,0,2)$& $a x^3 + b y = c y^2 z$&$(0,0,1)$\\\hline
	$a x^2 y + b x z = c t y z$&$(0,0,0,1)$& $a x^2 y + b t y z = c z^2$&$(1,1,2,0)$\\\hline
	$a x^2 y + b t x z = c y z$&$(0,1,1,0)$& $a x^2 y + b t x z = c z^2$&$(0,1,1,0)$\\\hline
	$a x^2 y + b x y z = c t z$&$(0,0,0,1)$& $a t^2 + b x^2 y = c x y z$&$(0,0,1,0)$\\\hline
	$a x^2 y + b z^2 = c x y z$&$(0,2,1)$& $a x^2 y + b y z = c t^2 z$&$(0,1,0,1)$\\\hline
	$a x^2 y + b x z = c t^2 z$&$(1,0,1,1)$& $a x^2 y + b x z = c t z^2$&$(0,0,0,1)$\\\hline
	$a x^2 y+ b y z^2 = c t x$&$(0,0,0,1)$& $a x^2 y+ b y z^2 = c t^2$&$(0,1,0,1)$\\\hline
	$a t y + b x^2 y = c x z^2$&$(0,1,1,0)$& $a t^2 + b x^2 y = c x z^2$&$(1,0,1,1)$\\\hline
	$a x^2 y + b x z^2 = c y^2$&$(1,3,2)$& $a t x + b x^2 y = c y^2 z$&$(0,0,1,0)$\\\hline
	$a x^2 y + b x z = cy^2 z$&$(1,1,2)$& $a t^2 + b x^2 y = c y^2 z$&$(0,0,1,0)$\\\hline
	$a x^2 y + b x^2 = cy^2 z$&$(0,0,1)$& $a x^2 y + b x y^2 = c z^2$&$(1,1,2)$\\\hline
	$a t y + b x^2 y = c x^2 z$&$(0,0,1,0)$& $a x^2 y + b x^2 z = c y z$&$(0,1,1)$\\\hline
	$a t^2 + b x^2 y = c x^2 z$&$(0,0,1,0)$& $a x^2 y + b x^2 z = c y^2$&$(0,1,1)$\\\hline
	$a x^3 + b x y z = c t y$&$(0,0,0,1)$& $a x^3 + b x y z = c y^2$&$(1,2,0)$\\\hline
	$a x^3 + b x y = c y z^2$&$(1,2,1,)$& $a x^3 + b t z = c y^2 z$&$(1,1,2,1)$\\\hline
	$a x^3 + b t y = c y^2 z$&$(0,0,1,0)$& $a x^3 + b y z = c y^2 z$&$(1,1,2)$\\\hline
	$a x^3 + b x y = c y^2 z$&$(0,0,1)$& $a x^3 + b y^2 z = c z^2$&$(3,2,5)$\\\hline
	$a x^3 + b y^2 = c y^2 z$&$(0,0,1)$& $a x^3 + b x y^2 = c y z$&$(0,0,1)$\\\hline
	$a x^3 + b x y^2 = c z^2$&$(1,1,2)$& $a x^3 + b x^2 y = c y z$&$(0,0,1)$\\\hline
	$a x^3 + b x^2 y = c z^2$&$(1,1,2)$& $a x^3 + by^3 = c x z$&$(0,0,1)$\\\hline
	$a x^3 + by^3 = cz^2$&$(1,1,2)$& $a x^3 + b y^3 = c x y z$&$(0,0,1)$ \\\hline
	$a x^2 y + b x y z = c t^2 z$&$(0,1,0,1)$& $a x^2 y + b x y z = c t z^2$&$(0,0,0,1)$\\\hline
	$a x^2 y + b y z^2 = c t x z$&$(0,0,0,1)$& $a t^2 y + b x^2 y = c x z^2$&$(0,1,1,0)$\\\hline
	$a x^2 y + b x z^2 = c t y^2$&$(0,0,0,1)$& $a x^2 y + b t x z = c y^2 z$&$(0,1,0,1)$\\\hline
	$a x^2 y + b x^2 z = c t y z$&$(0,0,0,1)$& $a t^2 y + b x^2 y = c x^2 z$&$(0,0,1,0)$\\\hline
	$a x^2 y + b x^2 z = c t y^2$&$(0,0,0,1)$& $a x^2 z+ b y^2 z = c x^2 y$&$(1,1,0)$\\\hline
	$x^3 + b x y z = c t^2 y$&$(1,2,0,1)$& $a x^3 + b x y z = c t y^2$&$(0,0,0,1)$\\\hline
	$a x^3 + b t y z = c y^2 z$&$(1,2,0,1)$& $a x^3 + b x y z = c y^2 z$&$(1,2,0)$\\\hline
	$a t^2 z+ b y^2 z = c x^3$&$(1,1,0,1)$& $a x^3 + b t^2 y = c y^2 z$&$(0,0,1,0)$\\\hline
	$a x^3 + b y^2 z = c t z^2$&$(0,0,0,1)$& $a x^3 + b t y^2 = c y^2 z$&$(0,0,1,0)$\\\hline
	$a x^3 + b x y^2 = c y^2 z$&$(0,0,1)$& $a x^2 y+ b y z^2 = c x^3$&$(1,0,1)$\\\hline
	$a x^3 + b x^2 y = c y^2 z$&$(0,0,1)$& $a x^3 + b z^3 = c x^2 y$&$(0,1,0)$\\\hline
\end{tabular}
	\caption{Families of cubic three-monomial equations solvable by Proposition \ref{prop:3monsuffcond}. \label{tab:3monsolfam}}
\end{center} 
\end{table}

After excluding some trivial families, e.g. when all monomials share a variable or equations in at most two variables, we identified $96$ families of three-monomial cubic equations, and for $88$ of them the condition of Proposition \ref{prop:3monsuffcond} holds, which implies that all these equations are automatically solvable. These families are listed in Table \ref{tab:3monsolfam}. The remaining $8$ families are listed in Table \ref{tab:notauto}. The first three families in Table \ref{tab:notauto} can be reduced by Theorem \ref{th:3monred} to quadratic two-variable equations of the form $Au^2+Bv^2+C=0$, which can then be solved by standard methods \cite[Section 10.5]{MR4368213}. The last $5$ families are homogeneous cubic three-variable equations that can be interpreted as elliptic curves in projective coordinates. Every integer solution $(x_0,y_0,z_0)$ to such equation with $\text{gcd}(x_0,y_0,z_0)=1$ produces an infinite family of solutions of the form $(x_0u, y_0u, z_0u)$, $u \in {\mathbb Z}$, which can be denoted as $(x_0 : y_0 : z_0)$ and interpreted as a (projective) rational point on the corresponding elliptic curve. To describe all such points, we need to compute the rank and the list of generators of its Mordell-Weil group. For the last problem, no general algorithm is known, but there are methods that work for a vast majority of specific curves with not-too-large coefficients \cite{MR1628193}.   

\begin{table}
\footnotesize
\begin{center}
\begin{tabular}{ | c | c |}
\hline
	\textbf{Equation } &  \textbf{Reduced by Theorem \ref{th:3monred} to}\\
	\hline\hline

	$a x+b x^2 y+c y z^2=0$&$Au^2+Bv^2+C=0$\\\hline
	
	$a x^2 y+b x z+c y z^2=0$&$Au^2+Bv^2+C=0$\\\hline
	
	$a x^2+b x^2 y+c y z^2=0$&$Au^2+Bv^2+C=0$\\\hline
	
	$a x^3+b y^3+c z^3=0$&$Au^3+Bv^3+Cw^3=0$\\\hline
	
	$a x^3+b y^2 z+c y z^2=0$&$Au^3+Bv^3+Cw^3=0$\\\hline
	
	$a x^3+b x y^2+c z^3=0$&$Au^6+Bv^3+Cw^2=0$\\\hline
	
	$a x^3+b x y^2+c y z^2=0$&$Au^4+Bv^4+Cw^2=0$\\\hline
		
	$a x^2 y+b y^2 z+c x z^2=0$&$Au^3+Bv^3+Cw^3=0$\\\hline
	\end{tabular}
		\caption{Families of cubic three-monomial equations that are not solvable by Proposition \ref{prop:3monsuffcond}. \label{tab:notauto}}
		\end{center}
	\end{table}

For quartic equations, we identified $60$ families that are not solvable by Proposition \ref{prop:3monsuffcond}. Table \ref{tab:quartic} lists these families, together with equations with  independent monomials to which they can be reduced to by Theorem \ref{th:3monred}. $31$ families are reduced to equations in one variable, two-variable equations of the form \eqref{eq:cympaxnpb}, or three-variable quadratic equations, which are easily solvable. $23$ families reduce to finding rational points on elliptic or super-elliptic curves, which is a difficult problem in general, but it can be solved by existing methods for many individual equations, especially with small coefficients.  Finally, $6$ families reduce to equations of the form 
$$
Au^n v^m+Bw^k+C=0,
$$ 
for integers $n,m,k\geq 2$ such that $\text{gcd}(n, m)=1$, most of which are difficult even for small coefficients.   
For example, a wide open conjecture of Schinzel and Tijdeman \cite{MR422150} predicts that for any polynomial $P(w)$ with integer coefficients with at least three simple zeros, there may exist at most finitely many integers $w$ such that $P(w)$ is a powerful number, that is, a number representable as $u^3v^2$ for integers $u,v$. If $P(w)$ consists of two monomials, this leads to some difficult three-monomial equations in $3$ variables. As a simple example, $P(w)=w^3+1$ leads to equation
\begin{equation}\label{eq:x3y2mz3m1}
u^3v^2=w^3+1
\end{equation}
that seems to be difficult. Examples of three-monomial quartic equations equivalent to \eqref{eq:x3y2mz3m1} are 
$
x^3y + yz^2 + z = 0
$
and 
$
x+x^2y^2 + z^3 = 0.
$

\begin{table}
\footnotesize
\begin{center}
\begin{tabular}{ |c|c||c|c|}
	\hline
	\textbf{Equation } &  \textbf{Reduced to}&\textbf{Equation } &  \textbf{Reduced to}\\
	\hline\hline
	
	$a x^3 y+b z^2+c x y z^2=0$&$Au^2+B=0$&$a x^3 y+b x y z^2+c z^3=0$&$Au^2+B=0$\\\hline
	$a y^2+b x^2 y z+c z^2=0$&$Au^2+B=0$&$a x y^2+b x^2 y z+c z^2=0$&$Au^3+B=0$\\\hline
	
	$a x+b x^2 y^2+c z^2=0$&$Au^2+Bv^2+C=0$&$a x^2 y^2+b x z+c z^2=0$&$Au^2+Bv^2+C=0$\\\hline
	$a x^2+b x^2 y^2+c z^2=0$&$Au^2+Bv^2+C=0$&$a y+b x^2 y^2+c x^2 z^2=0$&$Au^2+Bv^2+C=0$\\\hline	
	$a y^2+b x^2 y^2+c x^2 z^2=0$&$Au^2+Bv^2+C=0$&$a x^2 y+b x^2 y^2+c z^2=0$&$Au^2+Bv^2+C=0$\\\hline	
	$a x^4+b y^2+c y^2 z^2=0$&$Au^2+Bv^2+C=0$&$a x^4+b y^2 z+c y^2 z^2=0$&$Au^2+Bv^2+C=0$\\\hline
	$a x^4+b x y+c y^2 z^2=0$&$Au^2+Bv^2+C=0$&$a x^4+b x^2 y+c y^2 z^2=0$&$Au^2+Bv^2+C=0$\\\hline
	$a x^2 y^2+b x z+c t^2 z^2=0$&$Au^2+Bv^2+C=0$&$a x^3 y+b x^2 z+c y z^3=0$&$Au^3+Bv^3+C=0$\\\hline
	$a x^2 y^2+b y z+c x z^2=0$&$Au^3+Bv^3+C=0$&$a x^3 y+b z+c y^2 z^2=0$&$Au^3+Bw^3+C=0$\\\hline
	$a x^3 y+b x z+c y z^3=0$&$Au^3+Bv^3+C=0$&$a x^2+b x^3 y+c y z^3=0$&$Au^3+Bv^3+C=0$\\\hline
	$a x^3+b x^3 y+c y z^3=0$&$Au^3+Bv^3+C=0$&$a x^2 y^2+b y z+c x^2 z^2=0$&$Au^4+Bv^4+C=0$\\\hline
	$a x^2+b x^2 y^2+c y z^2=0$&$Au^4+Bv^2+C=0$&$a x^4+b y+c y^2 z^2=0$&$Au^4+Bv^2+C=0$\\\hline

	$a x^2 y^2+b x^2 z^2+c y^2 z^2=0$&$Au^2+Bv^2+Cw^2=0$&$a x^4+b y^2+c z^2=0$&$Au^2+Bv^2+Cw^2=0$\\\hline
	$a x^4+b x^2 y^2+c z^2=0$&$Au^2+Bv^2+Cw^2=0$&$a x^4+b t^2 y^2+c y^2 z^2=0$&$Au^2+Bv^2+Cw^2=0$ \\\hline
	$a t^2+b x^2 y^2+c x^2 z^2=0$&$Au^2+Bv^2+Cw^2=0$&$a t^2 y^2+b x^2 y^2+c x^2 z^2=0$&$Au^2+Bv^2+Cw^2=0$\\\hline
	$a x^4+b x^2 y^2+c y^2 z^2=0$&$Au^2+Bv^2+Cw^2=0$&$a x^3 y+b t^3 z+c y^2 z^2=0$&$Au^3+Bv^3+Cw^3=0$\\\hline
		$a x^3 y+b y^2+c z^3=0$&$Au^3+Bv^3+Cw^3=0$&$a x^3 y+b y^2 z+c z^2=0$&$Au^3+Bv^3+Cw^3=0$\\\hline
		$a x^3 y+b x^3 z+c y^2 z^2=0$&$Au^3+Bv^3+Cw^3=0$&$a x^2 y^2+b t^2 y z+c t x z^2=0$&$Au^3+Bv^3+Cw^3=0$\\\hline
	$a x^3 y+b x y^3+c z^2=0$&$Au^4+Bv^4+Cw^2=0$&$a x^4+b y^3 z+c y^2 z^2=0$&$Au^4+Bv^4+Cw^2=0$\\\hline
$a x^4+b y^4+c z^2=0$&$Au^4+Bv^4+Cw^2=0$&$a x^4+b y^2 z+c z^2=0$&$Au^4+Bv^4+Cw^2=0$\\\hline
	$a x^3 y+b x^2 z^2+c y^2 z^2=0$&$Au^4+Bv^4+Cw^2=0$&$a x^4+b x^2 y^2+c z^4=0$&$Au^4+Bv^4+Cw^2=0$\\\hline
		$a x^2 y^2+b t^2 y z+c x^2 z^2=0$&$Au^4+Bv^4+Cw^2=0$&$a x^4+b x y^3+c z^2=0$&$Au^6+Bv^3+Cw^2=0$\\\hline
			$a x^4+b x y^3+c y^2 z^2=0$&$Au^6+Bv^3+Cw^2=0$&$a x^4+b x^2 y^2+c y z^3=0$&$Au^6+Bv^6+Cw^3=0$\\\hline
		$a x^4+b y^4+c z^4=0$&$Au^4+Bv^4+Cw^4=0$&$a x^4+b x y^2 z+c y z^3=0$&$Au^5+Bv^5+Cw^5=0$\\\hline
	$a x^3 y+b y^3 z+c x^2 z^2=0$&$Au^5+Bv^5+Cw^5=0$&	$a x^3 y+b y^3 z+c x z^3=0$&$Au^7+Bv^7+Cw^7=0$\\\hline
    $a x^4+b y^4+c x y z^2=0$&$Au^8+Bv^8+Cw^2=0$&$a x^4+b y^3 z+c y z^3=0$&$Au^8+Bv^8+Cw^4=0$\\\hline	
	$a x^4+b x y^3+c y z^3=0$&$Au^9+Bv^9+Cw^3=0$&$a x^4+b x y^3+c z^4=0$&$Au^{12}+Bv^4+Cw^3=0$\\\hline
	$a x^2 y^2+b z+c x z^2=0$&$Au^3v^2+Bw^2+C=0$&$a y+b x^2 y^2+c x z^2=0$&$Au^3v^2+Bw^2+C=0$\\\hline
		$a x+b x^2 y^2+c z^3=0$&$Au^3v^2+Bw^3+C=0$&$a x^3 y+b z+c y z^2=0$&$Au^3v^2+Bw^3+C=0$\\\hline
	
	$a x+b x^3 y+c y z^2=0$&$Au^4v^3+Bw^2+C=0$&	$a x+b x^3 y+c y^2 z^2=0$&$Au^5v^4+Bw^2+C=0$\\\hline
\end{tabular}
	\caption{Families of quartic three-monomial equations that are not solvable by Proposition \ref{prop:3monsuffcond}.\label{tab:quartic}}
	\end{center}
	\end{table}

\begin{table}
	\footnotesize
	\begin{center}
		\begin{tabular}{ | c | c |c|c|c|c|}
			\hline
			\textbf{vars \textbackslash  degree} & \textbf{10}&\textbf{100}&\textbf{1000}&\textbf{10000}&\textbf{100000} \\\hline \hline	
			\textbf{3}&0.319&0.196&0.217&0.205&0.202 \\\hline
			\textbf{4}&0.553&0.473&0.469&0.502&0.466 \\\hline
			\textbf{5}&0.73&0.676&0.666&0.65&0.676 \\\hline
			\textbf{6}&0.854&0.796&0.824&0.802&0.823 \\\hline
			\textbf{7}&0.922&0.901&0.884&0.882&0.872 \\\hline
			\textbf{8}&0.955&0.952&0.937&0.934&0.935 \\\hline
			\textbf{9}&0.974&0.967&0.975&0.96&0.959 \\\hline
			\textbf{10}&0.987&0.984&0.979&0.978&0.979 \\\hline
		\end{tabular}
		\caption{\label{tab:random_equations} Empirical estimates for the proportion of families of three-monomial equations solvable by Proposition \ref{prop:3monsuffcond}.}
	\end{center} 
\end{table}

\subsubsection{Random equations}

What proportion of general three-monomial equations are solvable by Proposition \ref{prop:3monsuffcond}? To answer this question empirically, we study a simple model of a random three-monomial equation: fix the number of variables $n$ and integer $d>0$, and select all degrees $\alpha_i$, $\beta_i$ and $\gamma_i$ in \eqref{eq:gen3mon2} to be independent uniformly random integers on the interval $[0,d]$. We do not need to select coefficients $a,b,c$, because the condition of Proposition \ref{prop:3monsuffcond} does not depend on them. For each $3\leq n \leq 10$ and for $d=10^m$, $m=1,\dots,5$, we empirically estimated the proportion of equations satisfying the condition of Proposition \ref{prop:3monsuffcond}, see Table \ref{tab:random_equations}. As we can see, for a fixed number of variables $n$, the proportion of solvable equations decreases with $d$ but not fast, and seems to be approaching a limit: the data for $d=10^4$ and $d=10^5$ are similar. On the other hand, for any fixed $d$, this proportion quickly increases with $n$. For $d=10^5$, only about $20\%$ of $3$-variable equations satisfy the condition of Proposition \ref{prop:3monsuffcond}, but this proportion increases to almost $98\%$ for $10$-variable equations. The data suggests that if $n\to \infty$, then $100\%$ of the three-monomial equations are solvable by Proposition \ref{prop:3monsuffcond}.

\section{Conclusions}\label{sec:concl}

This paper develops the general method for solving all three-monomial two-variable Diophantine equations. The result may seem quite surprising, given that this family includes some equations like \eqref{eq:x4paxypy3} that were thought to be difficult. We are confident that this result is ``the best possible one could hope for'' without a major breakthrough, in the sense that it seems very difficult to resolve (i) all four-monomial equations in two variables or (ii) all three-monomial equations in three variables. Family (i) contains some well-known open equations, see \cite{grechuk2021diophantine, grechuk2022smallest}, and its algorithmic resolution looks no easier than the resolution of all two-variable Diophantine equations. Family (ii) contains, for example, equations of the form \eqref{eq:genFermat}, which are known as generalized Fermat equations and are the subject of current deep research \cite{bennett2016generalized}. 

Our algorithm for solving three-monomial two-variable Diophantine equations reduces any equation in this family, which is not covered by Theorem \ref{th:Runge}, to a finite number of equations of the form \eqref{eq:cympaxnpb} that can be solved by Theorem \ref{th:ellBaker1}. This algorithm is in general inefficient for two reasons. Firstly, the number of the resulting equations \eqref{eq:cympaxnpb} may grow rapidly with the coefficients of the original equation. Secondly, the known algorithms for solving individual equations of the form \eqref{eq:cympaxnpb} are quite slow in the worst case. However, the method is practical for solving individual three-monomial two-variable equations with not-too-large exponents and coefficients, see Tables \ref{tab:ntsx4paxypy3} and \ref{tab:xnpxkylpym} for some examples.

The main idea of the proof is the easy but extremely useful Lemma \ref{le:main}, that allows us to conclude that the exponents with which every prime $p$ enters the prime factorization of variables must satisfy one of three linear equations. Then these equations can be solved using the well-known theory of linear equations in non-negative integer variables. 

We have also applied the same idea to general three-monomial equations, in any number of variables. This allowed us to reduce any such equation to a finite number of equations with independent monomials. In many examples, the resulting equations are easy or well-known, hence the described reduction completely solves the original equations. An important example of a family of three-monomial equations that can be completely solved by the presented method are equations covered by Proposition \ref{prop:3monsuffcond}. Empirical results suggest that the number of three-monomial equations in this family approaches $100\%$ of all three-monomial equations in $n$ variables of any fixed degree if $n$ goes to infinity. Hence, our results are the most useful in the two ``opposite'' cases: $n=2$ and $n$ large. In this sense, the most difficult case is $n=3$ variables, where empirical data suggest that only about $20\%$ of the three-monomial equations are covered by Proposition \ref{prop:3monsuffcond}. The simplest-looking examples of difficult three-monomial equations in $3$ variables are the six families of quartic equations listed at the bottom of Table \ref{tab:quartic}, as well as quintic equation \eqref{eq:x3y2mz3m1} with independent monomials. 

\section*{Acknowledgements}

The authors are grateful to the referee for very detailed comments and suggestions, which helped to improve the quality of the paper. 

%% The Appendices part is started with the command \appendix;
%% appendix sections are then done as normal sections
%% \appendix

%% \section{}
%% \label{}

%% If you have bibdatabase file and want bibtex to generate the
%% bibitems, please use
%%
  \bibliographystyle{elsarticle-num} 
  \bibliography{references}

%% else use the following coding to input the bibitems directly in the
%% TeX file.

%\begin{thebibliography}{00}
%
%%% \bibitem{label}
%%% Text of bibliographic item
%
%\bibitem{}
%
%\end{thebibliography}
\end{document}